\documentclass[12pt]{article}
\usepackage{latexsym}
\usepackage{amssymb}
\input{xy}
\xyoption{all}

\usepackage{amscd}

\usepackage[utf8]{inputenc}

\usepackage{color}

\usepackage{soul}

\newtheorem{theorem}{THEOREM}[section]
\newtheorem{proposition}[theorem]{PROPOSITION}
\newtheorem{lemma}[theorem] {LEMMA}
\newtheorem{example}[theorem]{EXAMPLE}
\newtheorem{obser}[theorem]{REMARK}

\newtheorem{corollary}[theorem]{COROLLARY}

\newenvironment{proof}{\noindent {\it\bf Proof.} \rm}

\def\Hom{\mathop{\rm Hom}\nolimits}

\def\qed{\hfill \mbox{$\square$}}
\def\End{\mathop{\rm End}\nolimits}

\def\Ext{\mathop{\rm Ext}\nolimits}

\def\soc{\mathop{\rm soc}\nolimits}

\def\Im{\mathop{\rm Im}\nolimits}
\def\dim{\mathop{\rm dim_k}\nolimits}
\def\Der{\mathop{\rm Der}\nolimits}

\def\Int{\mathop{\rm Int}\nolimits}

\def\id{\mathop{\rm id}\nolimits}

\def\HH {{\rm HH}}

\def\path{\mathop{\rightsquigarrow}\nolimits}
\def\mod{\mathop{\rm mod}\nolimits}

\def \Z{{\mathbb Z}}
\def\mod{\mbox{{\rm mod}}}

\def\S{\mathcal S}

\def \Pe{{\mathbb P}}
\def\Po{\mathcal P}
\def\T{\mathcal T}
\def\Q{\mathcal Q}
\def\A{\mathbb A}

\newcommand{\QB}{\tilde Q}
\newcommand{\QC}{Q}

\newcommand{\ralf}{ }

\begin{document}

\title{On the first Hochschild cohomology group of a cluster-tilted algebra}

\author{Ibrahim Assem,  \  Mar\'\i a Julia Redondo  \ and Ralf Schiffler  \thanks{\footnotesize The first author gratefully acknowledges partial support from the NSERC of Canada, the FRQNT of Qu\'ebec and the Universit\'e de Sherbrooke, the second author has been supported by the project PICT-2011-1510
and is a  research member of
CONICET (Argentina), and the third author gratefully acknowledges partial support from the NSF Grant  DMS-1001637 and  the University of Connecticut. 
The authors wish to thank the { referees of this paper whose useful comments allowed to discover gaps in the proof.}}}

\date{}

\maketitle

\begin{center}
{\it To the memory of Dieter Happel}
\end{center}

\begin{abstract}
Given a cluster-tilted algebra $B$, we study its first Hochschild cohomology group ${\HH}^1(B)$ with coefficients in the $B$-$B$-bimodule $B$. If $C$ is a  tilted algebra such that $B$ is the relation-extension of $C$, then we show that if  $B$ is tame, then ${\HH}^1(B)$ is isomorphic, as a $k$-vector space, to the direct  sum of ${\HH}^1(C)$ with $k^{n_{B,C}}$, where $n_{B,C}$ is an invariant linking the bound quivers of $B$ and $C$. In the representation-finite case, ${\HH}^1(B)$ can be read off simply by looking at the quiver of $B$.
\end{abstract}

\small \noindent 2010 Mathematics Subject Classification : 13F60, 16E40, 16G20

\section{Introduction}

This paper is devoted to the study of the first Hochschild cohomology group ${\HH}^1(B)$ with coefficients in the $B$-$B$-bimodule $B$, see \cite{CE}. 

Cluster-tilted algebras were defined in \cite{BMR} and in \cite{CCS} for the type $\A$, as a by-product of the extensive theory of cluster algebras of Fomin and Zelevinsky \cite{FZ}. Now, it  has been shown in \cite{ABS1} that every cluster-tilted algebra $B$ is the relation-extension of a tilted algebra $C$. Our goal is to relate the Hochschild cohomologies of the two algebras $B$ and $C$. The main step in our argument consists in defining an equivalence relation between the arrows in the quiver of $B$ which are not in the quiver of $C$.  The number of equivalence classes is then denoted by $n_{B,C}$. 
The  first two authors have shown in \cite{AR} that, if the cluster-tilted algebra $B$ is schurian, then there is a short exact sequence of vector spaces 
\[ \xymatrix { 0 \ar[r] &  k^{n_{B,C}}  \ar[r] & {\HH}^1(B) \ar[r]^\varphi &  {\HH}^1(C) \ar[r] & 0.} \]
This holds true, for instance, when $B$ is representation-finite. In a further paper, \cite{ABIS}, it is actually proven that if $B$ is a cluster-tilted algebra (but generally not schurian) then $\varphi$ is still surjective and the kernel of the morphism $\varphi$ is equal to the first cohomology group ${\HH}^1(B,E)$ of $B$ with coefficients in the bimodule $E$.

Our objective in the present paper is to show that the result of \cite{AR} also holds true in case the cluster-tilted algebra $B$ is tame (that is, is of Dynkin or euclidean type).  

\begin{theorem}\label{thmA}
Let $k$ be an algebraically closed field and let $B$ be a tame cluster-tilted algebra which is the relation-extension of the tilted algebra $C$.  Then there is a short exact sequence of vector spaces 
$$0 \to  k^{n_{B,C}}  \to {\HH}^1(B) \to  {\HH}^1(C) \to 0.$$
\end{theorem}

We next show that, for any cluster-tilted algebra $B$, we have ${\HH}^1(B)=0$ if and only if $B$ is hereditary and its quiver is a tree, that is, $B$ is simply connected.  This answers positively for all cluster-tilted algebras Skowro\'nski's question in \cite[Problem 1]{S}{: For which algebras is simple connectedness equivalent to the vanishing of the first Hochschild cohomology group?}

Finally, we consider the case where the cluster-tilted algebra $B$ is representation-finite and show that the $k$-dimension of ${\HH}^1(B)$ can be computed simply by looking at the quiver of $B$: indeed, in this case, for any tilted algebra $C$ such that $B$ is { the } relation-extension of $C$, we have ${\HH}^1(C)=0$ and moreover the invariant $n_{B,C}$ does not depend on the particular choice of $C$ (and thus is denoted simply by $n_B$). Recalling that an arrow in the quiver of $B$ is called {\it inner} if it belongs to two chordless cycles, our theorem may be stated as follows.

\begin{theorem}\label{thmB}
 Let $B$ be a representation-finite cluster-tilted algebra. Then the dimension $n_B$ of ${\HH}^1(B)$ equals the number of chordless cycles minus the number of inner arrows in the quiver of $B$.
\end{theorem}

The paper is organized as follows. In section 2, after briefly setting the notation and recalling the necessary notions, we present results on systems of relations in cluster-tilted algebras. We then introduce the arrow equivalence relation in section 3.  In section 4 we describe the tilted algebras $C$ that are associated to the cluster-tilted algebra $B$ and section 5 is devoted to the proof of Theorem \ref{thmA}. Section 6  contains the proof of Theorem \ref{thmB}.

\section {Systems of relations.}

{\ralf Let $k$ be an algebraically closed field, then it is
well-known that any basic and connected finite dimensional $k$-algebra  $C$
can be written in the form $C=kQ/I,$ where $Q$ is a connected quiver, $kQ$
its path algebra and $I$ an admissible ideal of $kQ$.
  The pair $\left( Q,I\right) $ is then called a \textit{bound quiver}. We recall that finitely generated $C$-modules can be identified with
representations of the bound quiver $\left( Q,I\right) ,$ thus any such  module 
$M$ can be written as $M=\left( M(x),M(\alpha )\right) _{x\in Q_{0},\alpha
\in Q_{1}}$} (see, for instance,
\cite{ASS}).

 A \textit{relation} from $x \in Q_0$ to $y \in Q_0$ is a linear combination $\rho= \sum_{i=1}^{m} a_i w_i$ where each $w_i$ is a path of length at least two from $x$ to $y$ and $a_i\in k$ for each $i$.  If $m=1$ then $\rho$ is \textit{monomial}.  The relation $\rho$ is \textit{minimal} if each scalar $a_i$ is non-zero and $\sum_J a_i w_i \not \in I$ for any non-empty proper subset $J$ of the set $\{ 1, \dots, m \}$, and it is \textit{  strongly minimal} if each scalar $a_i$ is non-zero and $\sum_J b_i w_i \not \in I$ for any non-empty proper subset $J$ of the set $\{ 1, \dots, m \}$, where each $b_i$ is a non-zero scalar.
Two paths $u,v$ will be called {\it parallel} if they have the same source and target.  They are called {\it antiparallel} if the source (or the target) of $u$ equals the target (or the source, respectively) of $v$.

{We sometimes consider an algebra $C$ as a category, in
which the object class $C_{0}$ is a complete set $\left\{
e_{1},\ldots,e_{n}\right\} $ of primitive orthogonal idempotents of $C$ and  $C(x,y)=e_x C e_y$ is the
set of morphisms from $e_{x}$ to $e_{y}$.
 An algebra $C$ is  \textit{constricted} if, for any arrow from $x$ to $y$ in $Q_1$, we have $\dim e_x C e_y= 1$, see \cite{BM}.}

For a general background on the cluster category and cluster-tilting, we refer the reader to \cite{BMRRT}.
It is shown in \cite{ABS1} that, if $T$ is a tilting module over a hereditary algebra $A$, so that $C=\End_A(T)$ is a tilted algebra, then the trivial extension $\tilde C= C \ltimes \Ext^2_C(DC,C)$ (the \textit{relation-extension} of $C$) is cluster-tilted and, conversely, any cluster-tilted algebra is of this form (but in general, not uniquely: see \cite{ABS2}).  As a consequence, we have a description of the quiver of $\tilde C$.  Let $R$ be a \textit{system of relations} for the tilted algebra $C=kQ/I$, that is, $R$ is a subset of $\cup_{x,y \in Q_0}\, e_xIe_y$ such that
$R$, but no proper subset of $R$, generates $I$ as an ideal of $kQ$.  It is shown in \cite{ABS1} that the quiver  $\tilde Q$ of $\tilde C$ is as follows:
\begin{itemize}
\item[(a)] $\tilde Q_0 = Q_0$;
\item[(b)] For $x,y \in Q_0$, the set of arrows in $\tilde Q$ from $x$ to $y$ equals the set of arrows in $Q$ from $x$ to $y$ (which we call \textit{old arrows}) plus $\vert R \cap I(y,x)\vert$ additional arrows (which we call \textit{new arrows}).
\end{itemize}

The relations in $\tilde I$ are given by the partial derivatives of the potential $W = \sum_{\rho \in R} \alpha_{\rho} \rho$, with $\alpha_\rho$ the new arrow associated to the relation $\rho$, see \cite{Ke}.  

Now we   show that $R$ can be chosen as a set of  strongly minimal relations.

\begin{lemma}
If $\rho = \sum_{i=1}^m \lambda_i w_i \in  R$, with $\lambda_i \not = 0$, is not   strongly minimal, there exists $\rho' = \sum_{i=1}^m \mu_i w_i \in I$ with $\mu_1= \lambda_1$ which is  strongly minimal.
\end{lemma}

\begin{proof}
We proceed by induction on $m$.  If $m=2$ and $\rho=\lambda_1 w_1 + \lambda_2 w_2$ is not  strongly minimal, then it is clear that $w_1, w_2$ are relations in $I$ and hence we may take $\rho'= \lambda_1 w_1$.\\
Assume now $m > 2$ and $\rho $ is not strongly minimal.
 Then there is a relation
$\rho_1 = \sum_J \beta_j w_j \in I$, with $J$ a proper non-empty subset of $\{1, \dots, m\}$, $\beta_j \not = 0$. By induction on $m$ we may assume that $\rho_1$ is  strongly minimal. If $1 \in J$, we take $\rho'=\frac{\lambda_1}{\beta_1} \rho_1$ and we are done.
If $1 \not \in J$, let $s$ be the first element in $J$.  We apply the inductive hypothesis to the relation $\rho - \frac{\lambda_s}{\beta_s} \rho_1$. \qed
\end{proof}

\medskip

 A relation $\rho$ is called \textit{triangular} if it is a linear combination of paths that do not contain oriented cycles.

\begin{lemma}  \label{strongly}
Any system of  triangular relations $R=\{ \rho_1, \cdots, \rho_t \}$ can be replaced by a system of   strongly minimal relations $R'= \{ \rho'_1, \cdots, \rho'_t \}$. { Morevoer, each $\rho_i'$ is a linear combination of the paths which occur in the relation $\rho_i$.}
\end{lemma}
 
\begin{proof}
We proceed by induction on $t$.  If $t=1$, then $\rho= \sum_{i=1}^m \lambda_i w_i $ is already  strongly  minimal, since if it is not, then, by the previous lemma we get a relation $\rho'= \sum_{i=1}^m \mu_i w_i \in I$ with $\mu_1= \lambda_1$. Without loss of generality we may assume that $w_1$ has maximal length, and hence the relation $\rho - \rho' = \sum_{i=2}^m (\lambda_i -\mu_i) w_i$  belonging to the ideal generated by $\rho$  yields a contradiction:  in its  triangular expression as an element in $<\rho>$, there should be a summand of the form $\mu u_1 w_1 u_2$, with $\mu$ a non-zero scalar, $u_1, u_2$ paths in $Q$, and then $u_1 w_1 u_2$ is $w_1$ or a path of greater length, so this term cannot appear in $\rho-\rho'$.\\
Let $t>1$, let $\{ w_1, \dots, w_s \}$ be a complete set of paths appearing in the relations $\rho_i$, that is, 
$$\rho_i = \lambda_{1i} w_1 + \cdots + \lambda_{si} w_s.$$
Without loss of generality, we may assume that $w_1$ has maximal length and that $\lambda_{11} \not = 0$.
Now, the ideal generated by the set $\{ \rho_1, \cdots, \rho_t \}$ is equal to the ideal generated by the set
$$\{\rho_1, \tilde \rho_2, \cdots, \tilde \rho_t\}$$
with 
$$ \tilde \rho_j = \rho_j - \frac{\lambda_{1j}}{\lambda_{11}} \rho_1.$$
 If we apply the previous lemma to $\rho_1$ we get a  strongly minimal relation $\rho'_1$ with $\lambda_{11}$ as the first coefficient. Following an argument similar to what we did in the case  $t=1$, using the maximality of $w_1$ we get that the relation $\rho_1 - \rho'_1$ belongs to the ideal $ < \tilde \rho_2, \cdots, \tilde \rho_t>$, and so we get a system of relations $\{ \rho'_1, \tilde \rho_2, \cdots, \tilde \rho_t\}$, with $\rho'_1$  strongly minimal.  Now we proceed by induction on the set  $\{ \tilde \rho_2, \cdots, \tilde \rho_t\}$, and we get a system of  relations $ \{ \rho'_2, \cdots, \rho'_t \}$ which are  strongly minimal with respect to the ideal $ I'=< \rho'_2, \cdots, \rho'_t >$. Assume that one of these relations is not  strongly minimal with respect to $I$, say $\rho'_i= \sum_{i=2}^s \beta_i w_i$ and $\rho''= \sum_J \mu_i w_i \in I$, where $J$ is a proper subset of $\{ 2, \cdots, s\}$. So $\rho'' \not \in I'$ says that if we write it as an element in $I$, the relation $\rho'_1$ should appear.  Again we get a contradiction when considering the summands that contain $w_1$ as a subpath.
\qed
\end{proof}

\bigskip

Let $w$ be a nontrivial walk in a bound quiver $(Q,I)$.  Assume that one writes $w= uw'v$ where each of $u, w', v$ is a subwalk of $w$.  We say that $u,v$ {\it point to the same direction} in $w$ if $u$ and $v$, or $u^{-1}$ and $v^{-1}$, are paths in $Q$.

A reduced walk $w=uw'v$ having $u$ and $v$ pointing to the same direction  is called a {\it sequential walk} if there is a relation $\rho = \sum_i \lambda_i u_i$ such that $u=u_1$ or $u= u_1^{-1}$, there is a relation $\sigma = \sum_j \mu_j v_j$ such that $v=v_1$ or $v= v_1^{-1}$ and  no  subpath $w_1$ of $w'$, or of $(w')^{-1}$,  is involved in a relation of the form $\sum \lambda_i w_i$.

\begin{example}
The quiver
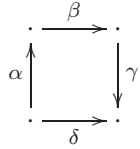
\[ \xymatrix{    
\cdot \ar[r]^\beta & \cdot \ar[d]^\gamma \\
\cdot \ar[u]^\alpha \ar[r]_\delta & \cdot
} \]
bound by the relation $\alpha \beta =0$ contains a sequential walk $w= \alpha \beta \gamma \delta^{-1} \alpha \beta$.

\end{example}

The following lemma generalises \cite{A,HL,AR}. 

\begin{lemma} \label{2.4}
Let $C=kQ/I$ be a tilted algebra.  Then the bound quiver $(Q,I)$ contains no sequential walk.
\end{lemma}

\begin{proof} 
Suppose that $w= u w' v$ is a sequential walk, with $u,v$ the (only) two subwalks of $w$ which are involved in relations pointing to the same direction. 
Clearly, $w'$ may have self-intersections and may also intersect the paths $u$ and $v$. Then $(Q,I)$ contains a bound subquiver $(Q',I')$, maybe not full, consisting of the points and arrows on $w$ as well as all the points and arrows which lie on a path parallel to $u$ or to $v$ and bound to it by a relation. This may be visualised as in the following pictures (drawn under the assumption that $w'$, the walk from $b_1$ to $b_s$, has no self-intersections and intersects neither $u$ nor $v$):

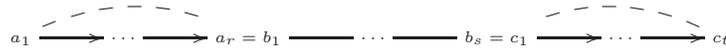
\begin{itemize}
\item [(a)]
\begin{tiny}
\[ \xymatrix{ a_1  \ar[r] \ar@/^1pc/@{--}[rr]     & \cdots  \ar[r] & a_r=b_1 \ar@{-}[r] & \cdots \ar@{-}[r] & b_s=c_1 \ar[r] \ar@/^1pc/@{--}[rr] & \cdots \ar[r] &  c_t
} \]
\end{tiny}

\item [(b)] 
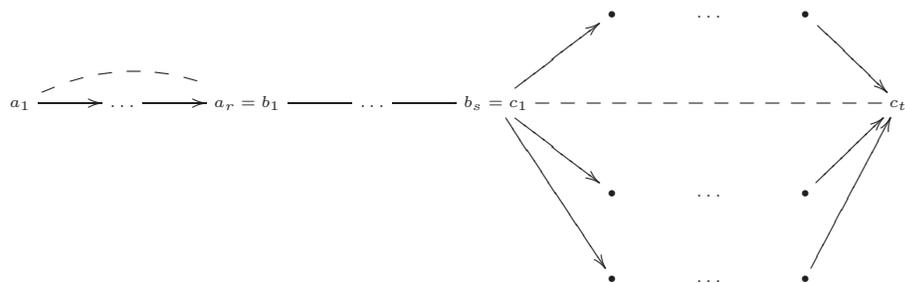
\begin{tiny}
\[ \xymatrix{  &&&&& \scriptscriptstyle \bullet & \dots & \scriptscriptstyle \bullet \ar[dr] \\
a_1  \ar[r] \ar@/^1pc/@{--}[rr]    & \dots \ar[r] & a_r=b_1 \ar@{-}[r] & \dots \ar@{-}[r] & b_s=c_1 \ar[ur] \ar[ddr] \ar[dr] \ar@{--}[rrrr] & & &  & c_t \\
&&&&& \scriptscriptstyle \bullet & \dots & \scriptscriptstyle \bullet \ar[ur] \\
&&&&& \scriptscriptstyle \bullet & \dots & \scriptscriptstyle \bullet \ar[uur]  } \]
\end{tiny} 

\item [(c)]
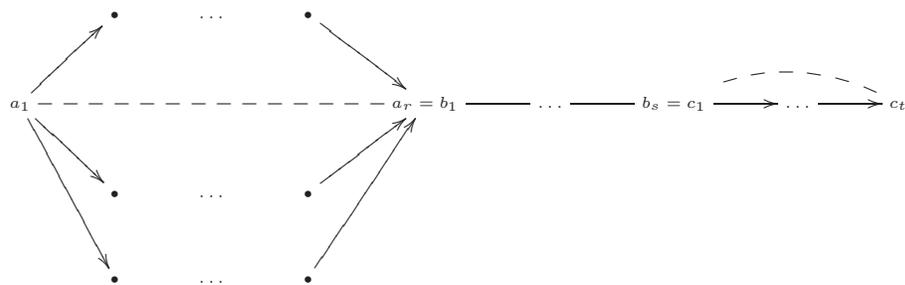
\begin{tiny}
\[ \xymatrix{  & \scriptscriptstyle \bullet & \dots & \scriptscriptstyle \bullet \ar[dr]   \\
a_1  \ar[ur] \ar[dr] \ar[ddr] \ar@{--}[rrrr] & & &  & a_r=b_1 \ar@{-}[r]    & \dots \ar@{-}[r] & b_s=c_1 \ar[r] \ar@/^1pc/@{--}[rr]    & \dots \ar[r] & c_t \\
& \scriptscriptstyle \bullet & \dots & \scriptscriptstyle \bullet \ar[ur]  \\
& \scriptscriptstyle \bullet & \dots & \scriptscriptstyle \bullet \ar[uur]  
} \]
\end{tiny}

\item [(d)]
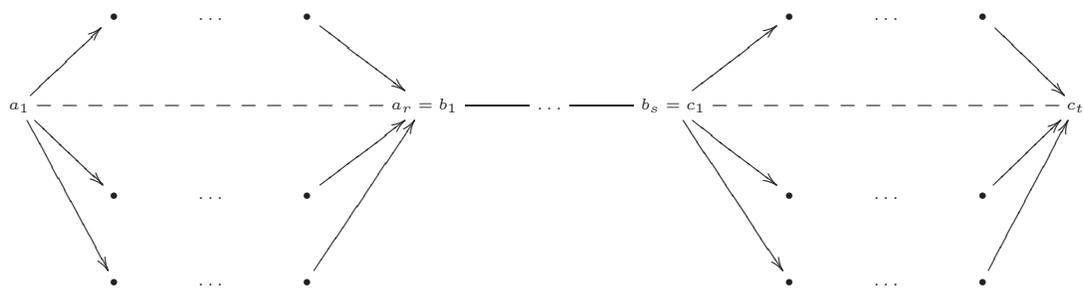
\begin{tiny}
\[ \xymatrix{  & \scriptscriptstyle \bullet &  \dots & \scriptscriptstyle \bullet \ar[dr] &&& & \scriptscriptstyle \bullet & \dots & \scriptscriptstyle \bullet \ar[dr] \\
a_1  \ar[ur] \ar[dr] \ar[ddr] \ar@{--}[rrrr] & & &  & a_r=b_1 \ar@{-}[r]    & \dots \ar@{-}[r] & b_s=c_1 \ar[ur] \ar[dr]  \ar[ddr] \ar@{--}[rrrr] & & &  & c_t \\
& \scriptscriptstyle \bullet & \dots & \scriptscriptstyle \bullet \ar[ur] &&& & \scriptscriptstyle \bullet & \dots & \scriptscriptstyle \bullet \ar[ur] \\
& \scriptscriptstyle \bullet & \dots & \scriptscriptstyle \bullet \ar[uur] &&& & \scriptscriptstyle \bullet & \dots & \scriptscriptstyle \bullet \ar[uur]
 } \]
\end{tiny}

\item [(e)]
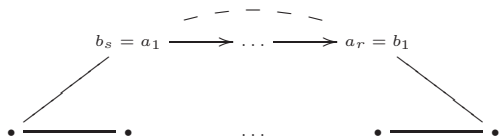
\begin{tiny}
\[ \xymatrix{ & b_s=a_1 \ar@/^1pc/@{--}[rr] \ar[r] & \dots \ar[r] & a_r=b_1 \ar@{-}[dr] \\
\scriptscriptstyle \bullet \ar@{-}[r] \ar@{-}[ru]  & \scriptscriptstyle \bullet & \dots & \scriptscriptstyle \bullet \ar@{-}[r] & \scriptscriptstyle \bullet} \]
\end{tiny}

\item [(f)]
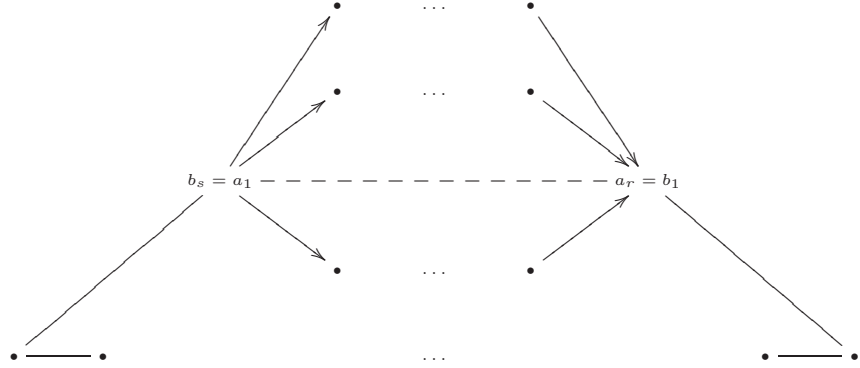
\begin{tiny}
\[ \xymatrix{  
& & & \scriptscriptstyle \bullet & \dots & \scriptscriptstyle \bullet \ar[ddr]  \\
&  & & \scriptscriptstyle \bullet & \dots & \scriptscriptstyle \bullet \ar[dr] \\
& & b_s=a_1 \ar[ur] \ar[uur] \ar[dr] \ar@{--}[rrrr] & & &   & a_r=b_1 \ar@{-}[ddrr] \\
& & & \scriptscriptstyle \bullet & \dots & \scriptscriptstyle \bullet \ar[ur] \\
\scriptscriptstyle \bullet \ar@{-}[r] \ar@{-}[rruu]  & \scriptscriptstyle \bullet &  &  & \dots &  & & \scriptscriptstyle \bullet \ar@{-}[r] & \scriptscriptstyle \bullet} \]
\end{tiny}
\end{itemize}
where we have represented relations by dotted lines.  The last two cases occur when $u=v$.   Moreover, the minimality of the length of $w'$ implies that there is no additional arrow between two $b_j$'s.
Let $C'=kQ'/I'$.  Because of \cite[III.6.5, p. 146]{Ha1}, $C'$ is also a tilted algebra.  In each of these cases above, let $M$ be the $C'$-module defined as a representation by
\[ M(x)=  \left\{ 
\begin{array}{rll}
k,  &\mbox{if $x$ is a point of the walk $w'$},\\
0, & \mbox{otherwise},
\end{array} \right. \]
and
\[ M(\alpha)= \left\{ 
\begin{array}{rll}
\id,  &\mbox{if $s(\alpha), t(\alpha) $ are points of the walk $w'$},\\
0, & \mbox{otherwise},
\end{array} \right. \]
for every point $x$ and arrow $\alpha$ in the quiver of $C'$.  Since there is no subpath $w_1$ of $w'$, or of $(w')^{-1}$, involved in a relation of the form $\sum \lambda_i w_i$, then $M$ is indeed a module.  It is actually a tree module, and therefore it is indecomposable, see \cite{K}. On the other hand, it can be seen that both of its projective and its injective dimensions equal two, a contradiction because $C'$ is tilted.
\qed
\end{proof}

\bigskip

Let now $\tilde C = k \tilde Q / \tilde I$ be the relation-extension of a tilted algebra $C$.  A walk $w= \alpha w' \beta$ in $(\tilde Q, \tilde I)$ is called a $C$-{\it sequential walk} if:

\begin{itemize}
\item [(i)] $w'$ consists entirely of old arrows,
\item [(ii)] $\alpha, \beta$ are two new arrows corresponding respectively to old relations $\rho= \sum_i \lambda_i u_i$ and $\sigma= \sum_j \mu_j v_j$, and
\item [(iii)] $w= u_i w' v_j$ is a sequential walk in $(Q,I)$ for any $i,j$.
\end{itemize}

\begin{corollary}\label{13}
Let $C=kQ/I$ be a tilted algebra.  Then the bound quiver of its  relation-extension $\tilde C$  contains no { $C$-sequential}  walk. 
\end{corollary} 

\section{Arrow equivalence}

The following lemma is an easy consequence of the main result in \cite{ACT}.  For the benefit of the reader, we give an independent proof. 

Recall from \cite{DWZ} that for a given arrow $\beta$, the \emph{cyclic partial derivative} $\partial_\beta$ in $\beta$ is defined on each cyclic path $\beta_1\beta_2\cdots\beta_s$ by $\partial_\beta (\beta_1\beta_2\cdots\beta_s) = \sum_{i:\beta=\beta_i} \beta_{i+1}\cdots\beta_s\beta_1\cdots\beta_{i-1}$. Note that $\partial_\beta(\beta_1\beta_2\cdots\beta_s)=\partial_\beta(\beta_j\cdots\beta_s\beta_1\cdots\beta_{j-1})$ for every $j$ such that $1 \leq j \leq s$.

From now on, we consider a given presentation of a tilted algebra $C=kQ/I$ with  minimal system of relations $R$ and consider its relation-extension $B=k \tilde Q / \tilde I$ { together with} the presentation having as relations the cyclic partial derivatives of the potential $W= \sum_{\rho \in R} \alpha_\rho \rho$, with $\alpha_\rho$ the new arrow associated to the relation $\rho$.  All these relations are triangular. Then we reduce this system to an equivalent system of strongly minimal relations each of which is a linear combination of the relations obtained from the partial derivatives (see Lemma \ref{strongly}.)

\begin{lemma}\label{relations}
{ Let $C=kQ/I$ be a tilted algebra and $B=k \tilde Q / \tilde I$ be such that $B= \tilde C$ and $\tilde I$ is generated by the partial derivatives of the potential. } 
Let $\rho= \sum_{ i = 1}^m a_i w_i$ be a minimal relation in $\tilde I$.  Then either $\rho$ is a relation in $I$, or { there exist $m$ new arrows} $\alpha_1, \dots, \alpha_m$ such that $w_i=u_i\alpha_iv_i$ (with $u_i,  v_i$ paths consisting entirely of old arrows).
\end{lemma} 

\begin{proof} Let  $\rho_1 , \ldots ,\rho_s$ be a system of minimal relations for the tilted algebra $C$. Then each relation $\rho_i$ induces a new arrow $\alpha_i$ and the product  $\rho_i\alpha_i$ is a linear combination of cyclic paths in the quiver of the cluster-tilted algebra $B$. The potential of $B$ can be given as $W= \sum_{i=1}^s  \rho_i \alpha_i$ 
and the ideal of $B$ is generated by all partial derivatives $\partial_\beta W$ of the potential $W$ {\ralf with respect to} the arrows $\beta$.
If $\beta$ is one of the new arrows $\alpha_i$ 
{\ralf then $\partial_\beta W $ is just the ``old" relation $\rho_i\in I$.}

If $\beta$ is an old arrow then 
$\partial_\beta W$ is a sum of terms which are cyclic permutations of $(\partial_\beta \rho_i ) \alpha_i$.  Now,  each of the summands contains exactly one new arrow $\alpha_i$. \qed
\end{proof}

\bigskip

The previous lemma asserts that if $\rho = \sum_i a_i w_i$ is a  strongly minimal relation lying in $\tilde I$ but not in $I$, then on each $w_i$ lies exactly one new arrow $\alpha_i$ and each new arrow appears in this way.  Clearly, the $\alpha_i$ are not necessarily distinct as arrows of $\tilde Q$. 

Lemma \ref{relations} above brings us to our main definition.  Let $B=k \tilde Q /\tilde I$ be a cluster-tilted algebra and $C=kQ/I$ a tilted algebra such that $B=\tilde C$.  We define a relation $\sim$ on the set $\tilde Q_1 \setminus Q_1$ of new arrows as follows. For every $\alpha \in \tilde Q_1 \setminus Q_1$, we set $\alpha \sim \alpha$.   If $\rho=\sum_{i=1}^m a_i w_i $ is a  strongly minimal  relation in $\tilde I$ and $\alpha_i $ are as in Lemma \ref{relations} above, then we set $\alpha_i \sim \alpha_j$ for any $i,j$ such that $1 \leq i, j \leq m$.

By Corollary \ref{13}, the relation $\sim$ is unambiguously defined.  It is clearly reflexive and symmetric.  We let $\approx$ be the least equivalence relation defined on the set $\tilde Q_1 \setminus Q_1$ such that $\alpha \sim \beta$ implies $\alpha \approx \beta$ (that is, $\approx$ is the transitive closure of $\sim$).

We define the \textit{relation invariant} of $B$ to be the number $n_{B,C}$ of equivalence classes under the relation $\approx$.

Observe that the equivalence relation $\approx$ is related to the direct sum decomposition of the $C$-$C$- bimodule $E$.  Indeed, $E$ is generated as $C$-$C$-bimodule by the new arrows.  If two new arrows occur in a  strongly minimal relation, this means that they are somehow yoked together in $E$.  It is proven in \cite[Lemma 4.3]{ABIS} that $E$ decomposes, as $C$-$C$-bimodule, into the direct sum of $n_{B,C}$ summands.

The following two lemmata will be useful in section \ref{sect 5}. They use essentially the fact that cluster-tilted algebras of type $\tilde\mathbb{A}$ are gentle (because of \cite[Lemma 2.5]{ABCP}) and in particular all relations are monomial of length 2 contained inside 3-cycles that is, cycles of the form
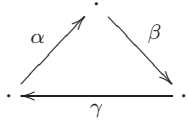
\[\xymatrix{&\cdot\ar[rd]^\beta\\
\cdot\ar[ru]^\alpha&&\cdot\ar[ll]^\gamma}
\]
bound by $\alpha\beta=\beta\gamma=\gamma\alpha=0$.

\begin{lemma}\label{lem 3.3}
 Let $B$ be a cluster-tilted algebra of type $\tilde\mathbb{A}$ and let $C_1,C_2$ be tilted algebras such that $B=\tilde C_1=\tilde C_2$. Let $R_1$, $R_2$ be systems of relations for $C_1$, $C_2$ respectively. Then $|R_1|=|R_2|$. 
\end{lemma}

\begin{proof}
 Indeed, in order to obtain $C_1$ and $C_2$ from $B$, we have to delete exactly one arrow from each chordless cycle (for the notion of chordless cycle, see \cite{BGZ} or section \ref{sect 6} below). Because $B$ is of type $\tilde\mathbb{A}$, then the chordless cycles are 3-cycles, and no arrow belongs to two distinct 3-cycles. Deleting exactly one arrow from each 3-cycle leaves a path of length 2. The system of relations for the tilted algebra consists in exactly these paths of length 2. This implies the statement.
 \qed 
\end{proof}

\medskip

\begin{lemma}
 \label{lem 5.5} Let $B=\tilde C$, where $C$ is a tilted algebra of type $\tilde\mathbb{A}$. Let $R$ be a system of relations for $C$. Then $n_{B,C}=|R|$. In particular, $n_{B,C}$ does not depend on the choice of $C$.
\end{lemma}
\begin{proof}
 Let $\alpha_i$, $\alpha_j$ be two equivalent new arrows, then there exists a sequence of new arrows 
 \[\alpha_i=\beta_1\sim\beta_2\sim\cdots\sim\beta_t=\alpha_j\]
 where $\beta_\ell$, $\beta_{\ell+1}$ appear in the  same  strongly minimal relation in $(\tilde Q,\tilde I)$. Now, $B$ is gentle. Hence  strongly minimal relations contain just one monomial. Therefore $\beta_\ell=\beta_{\ell+1}$ for each $\ell$, and $\alpha_i=\alpha_j$. This shows that the relation invariant $n_{B,C}$ is equal to the number of new arrows, and the latter is equal to $|R|$ because of \cite[Theorem 2.6]{ABS1}.\qed
\end{proof}

\section{The tame cluster-tilted algebras}\label{sect 5}
Our objective in
this section is to describe the tilted algebras $C$ that are associated to the tame cluster-tilted algebra $B$. Because of \cite[Theorem A]{BMR}, the tame representation-infinite cluster-tilted algebras
are just the cluster-tilted algebras of euclidean type, that is, the
relation-extensions of the tilted algebras of euclidean type. Our strategy will consist of reducing the proof to the case where $C$ is a constricted algebra. 

An algebra $K$ is {\it tame concealed} if there exists a tame hereditary algebra $A$ and a postprojective tilting $A$-module $T$ such that $K= \End_A(T)$.  Then $\Gamma(\mod K)$ consists of a postprojective component $\Po_K$, a preinjective component $\Q_K$ and a family $\T_K = (\T_\lambda)_{\lambda \in \Pe_1(k)}$ of stable tubes separating $\Po_K$ from $\Q_K$, see \cite[4.3]{Ri}.

We now define {\it tubular extensions} of a tame concealed algebra. A {\it branch} $L$ with a {\it root} $a$ is a finite connected full bound subquiver, containing $a$, of the following infinite tree, { where all compositions $\alpha \beta$ of two arrows labeled as $\alpha$ and $\beta$ are zero.}

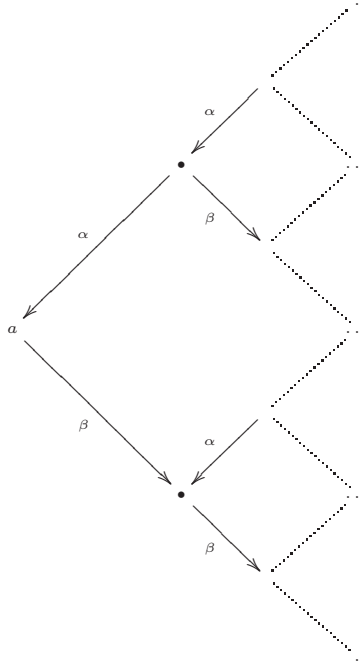
\begin{tiny}
\[ \xymatrix{ &&&& . \\
&&& \ar@{.}[ru] \ar@{..}[rd] \ar[ld]_\alpha\\ 
&& \scriptscriptstyle \bullet \ar[lldd]_\alpha \ar[rd]_\beta && \dots\\
&&& \ar@{..}[ru] \ar@{..}[rd]  \\ 
a \ar[rrdd]_\beta &&&& \dots  \\
&&& \ar@{..}[ru] \ar@{..}[rd] \ar[ld]_\alpha\\ 
&& \scriptscriptstyle \bullet \ar[rd]_\beta  && \dots\\
&&& \ar@{..}[ru] \ar@{..}[rd] \\ 
&&&& .
} \]
\end{tiny} 
Let now $K$ be a tame concealed algebra, and $(E_i)^n_{i=1}$ be a family of simple regular $K$-modules.  For each $i$, let $L_i$ be a branch with root $a_i$.  The tubular extension $B=K[E_i, L_i]_{i=1}^n$ has as objects those of $K, L_1, \cdots, L_n$ and as morphism spaces 
$$ B(x,y) =  \left\{   \begin{array}{lll}
K(x,y)  \quad &  \mbox{if $x,y \in K_0$}  \\
L_i(x,y)  & \mbox{if $x,y \in (L_i)_0$}  \\
L_i(x,a_i) \otimes_k E_i(y)  & \mbox{if $x \in (L_i)_0, y \in K_0$}  \\
0  & \mbox{otherwise.}
\end{array} \right . $$
The {\it tubular coextension} ${}^n_{i=1}[E_i, L_i]K$ is defined dually. \\
For each $\lambda \in \Pe_1(k)$, let $r_\lambda$ denote the rank of the stable tube $\T_\lambda$ of $\Gamma(\mod K)$.  The {\it tubular type} $n_B=(n_\lambda)_{\lambda \in  \Pe_1(k)}$ of $B$ is defined by $$n_\lambda = r_\lambda + \sum_{E_i \in \T_\lambda} \vert (L_i)_0 \vert.$$
Since all but at most finitely many $n_\lambda$ equal $1$, we write for $n_B$ the finite sequence containing at least two $n_\lambda$, including all those larger than $1$, in non-decreasing order.  We say that $n_B$ is {\it domestic} if it is one of the forms $(p,q), (2,2,r), (2,3,3), (2,3,4), (2,3,5)$.
The following structure theorem is
due to Ringel, see \cite[Theorem 4.9, p. 241]{Ri}.

\begin{theorem} Let $C$ be a representation-infinite tilted algebra of euclidean type.  Then $C$ contains a unique tame concealed full convex subcategory $K$ and $C$ is a domestic tubular extension or a domestic tubular coextension of $K$.
\end{theorem}

As a consequence of Ringel's theorem, we obtain the following.

\begin{lemma}\label{lem 5.2}
Let $C$ be a tilted algebra of euclidean type which is not constricted.  Then $C$ is given by one of the following two bound quivers, or their duals.

\begin{itemize}
\item[(1)]

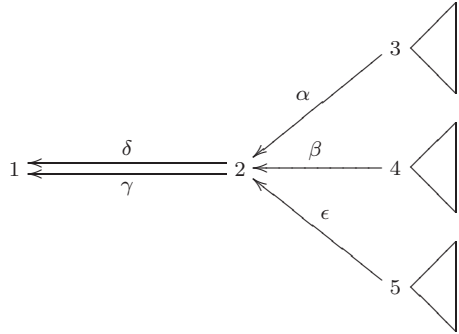
\[
\xy
\xymatrix"m"@C5pt@R5pt{{\scriptstyle3}\!\!&\!\! {\xy 
(0,0)*{}; (6,6)*{} **\dir{-};
(0,0)*{}; (6,-6)*{} **\dir{-};
(6,6)*{}; (6,-6)*{} **\dir{-};\endxy 
}\\
{\scriptstyle4}\!\!&\!\! {\xy 
(0,0)*{}; (6,6)*{} **\dir{-};
(0,0)*{}; (6,-6)*{} **\dir{-};
(6,6)*{}; (6,-6)*{} **\dir{-};\endxy 
}\\
{\scriptstyle5}\!\!&\!\! {\xy 
(0,0)*{}; (6,6)*{} **\dir{-};
(0,0)*{}; (6,-6)*{} **\dir{-};
(6,6)*{}; (6,-6)*{} **\dir{-};\endxy 
}
}
\POS-(20,16)
\xymatrix"n"{
   \scriptstyle 2
\ar@{<-}["m"]^\alpha  \ar@{<-}["m"d]^\beta  \ar@{<-}["m"dd]^\epsilon&
}
\POS-(30,0)
\xymatrix{\scriptstyle 1 
  \ar@<2pt>@{<-}["n"]^\delta\ar@<-2pt>@{<-}["n"]_\gamma
}
\endxy
\]

where the triangles are branches, possibly empty, bound by $\alpha \delta = 0, \beta \delta = \beta \gamma, \epsilon \gamma = 0$, and the branch relations.
\item[(2)]
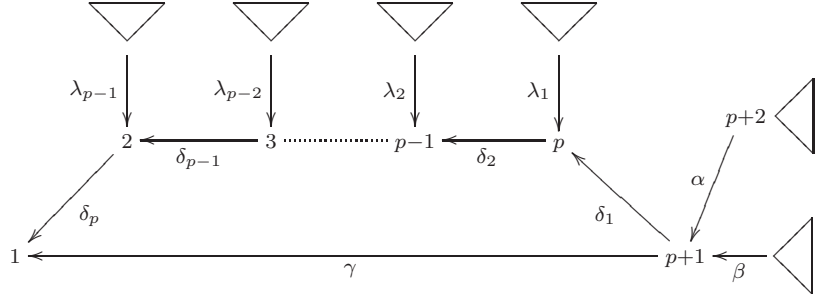
\[ \xy\xymatrix"m"@C5pt@R5pt{{\scriptstyle p+2}\!\!&\!\! {
\xy 
(0,0)*{}; (5,5)*{} **\dir{-};
(0,0)*{}; (5,-5)*{} **\dir{-};
(5,5)*{}; (5,-5)*{} **\dir{-};\endxy 
}} 
\POS-(96.5,-10)
\xymatrix@C20pt@R20pt{&\ar[d]_{\lambda_{p-1}}{\xy 
(0,0)*{}; (5,5)*{} **\dir{-};
(0,0)*{}; (-5,5)*{} **\dir{-};
(-5,5)*{}; (5,5)*{} **\dir{-};\endxy 
} &
{\xy 
(0,0)*{}; (5,5)*{} **\dir{-};
(0,0)*{}; (-5,5)*{} **\dir{-};
(-5,5)*{}; (5,5)*{} **\dir{-};\endxy 
}\ar[d]_{\lambda_{p-2}} &
{\xy 
(0,0)*{}; (5,5)*{} **\dir{-};
(0,0)*{}; (-5,5)*{} **\dir{-};
(-5,5)*{}; (5,5)*{} **\dir{-};\endxy 
}\ar[d]_{\lambda_{2}}
&
{\xy 
(0,0)*{}; (5,5)*{} **\dir{-};
(0,0)*{}; (-5,5)*{} **\dir{-};
(-5,5)*{}; (5,5)*{} **\dir{-};\endxy 
}\ar[d]_{\lambda_{1}}
&&~~
\\
 &  \scriptstyle 2\ar[dl]^{\delta_p} 
& \scriptstyle 3\ar[l]^{\delta_{p-1}}& \scriptstyle p-1\ar@{.}[l]  & \scriptstyle p\ar[l]^{\delta_2} \\
  \scriptstyle 1&   &  &  &  &   \scriptstyle p+1 \ar[lllll]^\gamma \ar[ul]^{\delta_1}
\ar@{<-}["m"uulllll]^\alpha&
{\xy 
(0,0)*{}; (5,5)*{} **\dir{-};
(0,0)*{}; (5,-5)*{} **\dir{-};
(5,5)*{}; (5,-5)*{} **\dir{-};\endxy 
}\ar[l]^\beta 
}
\endxy
\]
where the   triangles are branches, possibly empty, bound by $\alpha \delta_1  \cdots \delta_p=\alpha \gamma, \beta \gamma = 0, \lambda_i \delta_{i+1}=0$ for all $i$ such that $1 \leq i < p$, and the branch relations.
\end{itemize}
\end{lemma}

\begin{proof}
Assume $C$ is a tilted algebra of euclidean type which is not constricted.  Then there exists an arrow $\gamma: x \to y$ such that $\dim C(x,y) \geq 2$.  Since $C$ is tame, we actually have $\dim C(x,y)= 2$.  In particular, $C$ is representation-infinite.  Applying Ringel's theorem, we get that $C$ is, up to duality, a domestic tubular extension of a unique tame concealed full convex subcategory $K$ of $C$.  On the other hand, let $K'$ be the convex envelope of the points $x,y$ in $C$.  Then $K'$ is of the form

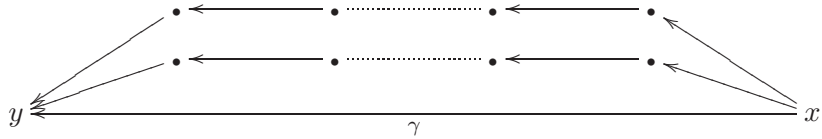
\[ \xymatrix@R10pt@C50pt{ & \scriptscriptstyle \bullet \ar[ddl] &  \ar[l]\scriptscriptstyle \bullet &\scriptscriptstyle \bullet \ar@{.}[l]  & \scriptscriptstyle \bullet\ar[l]\\
& \scriptscriptstyle \bullet \ar[dl] &  \ar[l]\scriptscriptstyle \bullet &\scriptscriptstyle \bullet \ar@{.}[l]  & \scriptscriptstyle \bullet\ar[l] \\
y   &  &  &  & & x \ar[lllll]^\gamma \ar[uul] \ar[ul]} \]
with $\dim K'(x,y)=2$.  Note that $K'$ is a full convex subcategory of $C$, hence it is tilted (because of \cite[III.6.5 p.146]{Ha1}).  Applying Lemma \ref{2.4} to $K'$, we deduce that $K'$ is of the form

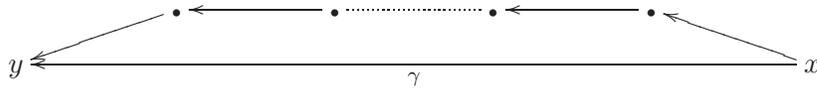
\[ \xymatrix@R10pt@C50pt{ & \scriptscriptstyle \bullet \ar[dl] &  \ar[l]\scriptscriptstyle \bullet &\scriptscriptstyle \bullet \ar@{.}[l]  & \scriptscriptstyle \bullet\ar[l] \\
y   &  &  &  & & x \ar[lllll]^\gamma  \ar[ul]} \]
Since $K'$ is hereditary, we get that $K'=K$.  The statement now follows by considering the possible branch extensions of $K$.
\qed \end{proof}

\medskip

\begin{lemma}\label{lem 5.3} Let $B$ be a cluster-tilted algebra of euclidean type. Assume that there exists no constricted tilted algebra $C$ such that $B=\tilde C$.
 Then $B$ is a cluster-tilted algebra of type $\tilde{\mathbb{A}}$ of one of the following forms or their duals:
 
 \begin{itemize}
\item[\textup{(i)}] 
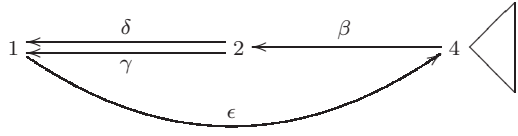
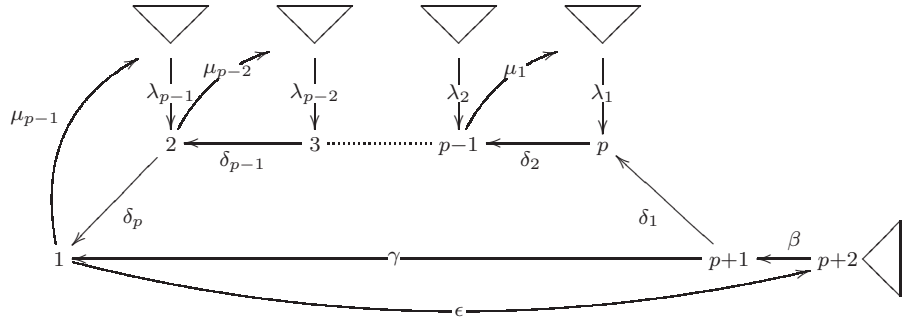
\[
\xy
\xymatrix"m"@C5pt@R5pt{
{\scriptstyle4}\!\!&\!\! {\xy 
(0,0)*{}; (6,6)*{} **\dir{-};
(0,0)*{}; (6,-6)*{} **\dir{-};
(6,6)*{}; (6,-6)*{} **\dir{-};\endxy 
}}\\
\POS-(28,0)
\xymatrix"n"{
   \scriptstyle 2 \ar@{<-}["m"]^\beta  &
}
\POS-(30,0)
\xymatrix{\scriptstyle 1 \ar@/_30pt/["m"]^\epsilon 
  \ar@<2pt>@{<-}["n"]^\delta\ar@<-2pt>@{<-}["n"]_\gamma
}
\endxy
\]

\item[\textup{(ii)}]
\[ \xy

\xymatrix@C20pt@R20pt{&\ar[d]|{\lambda_{p-1}}{\xy 
(0,0)*{}; (5,5)*{} **\dir{-};
(0,0)*{}; (-5,5)*{} **\dir{-};
(-5,5)*{}; (5,5)*{} **\dir{-};\endxy 
} &
{\xy 
(0,0)*{}; (5,5)*{} **\dir{-};
(0,0)*{}; (-5,5)*{} **\dir{-};
(-5,5)*{}; (5,5)*{} **\dir{-};\endxy 
}\ar[d]|{\lambda_{p-2}} &
{\xy 
(0,0)*{}; (5,5)*{} **\dir{-};
(0,0)*{}; (-5,5)*{} **\dir{-};
(-5,5)*{}; (5,5)*{} **\dir{-};\endxy 
}\ar[d]|{\lambda_{2}}
&
{\xy 
(0,0)*{}; (5,5)*{} **\dir{-};
(0,0)*{}; (-5,5)*{} **\dir{-};
(-5,5)*{}; (5,5)*{} **\dir{-};\endxy 
}\ar[d]|{\lambda_{1}}
&&~~
\\
 &  \scriptstyle 2\ar[dl]^{\delta_p} \ar@/^10pt/[ur]|{\mu_{p-2}}
& \scriptstyle 3\ar[l]^{\delta_{p-1}}& \scriptstyle p-1\ar@{.}[l]  \ar@/^10pt/[ur]|{\mu_{1}} & \scriptstyle p\ar[l]^{\delta_2} \\
  \scriptstyle 1\ar@/_20pt/[rrrrrr]|{\,\epsilon\,}
  \ar@/^20pt/[ruu]^{\mu_{p-1}}&   &  &  &  &   \scriptstyle p+1 \ar[lllll]|\gamma \ar[ul]^{\delta_1}
&
{{\scriptstyle p+2}\,\xy 
(0,0)*{}; (5,5)*{} **\dir{-};
(0,0)*{}; (5,-5)*{} **\dir{-};
(5,5)*{}; (5,-5)*{} **\dir{-};\endxy 
}\ar[l]_\beta 
}
\endxy
\]
\end{itemize}
 where the triangles are cluster-tilted algebras of type $\mathbb{A}$, possibly empty, bound by $\beta\gamma=0$, $ \gamma\epsilon=0$, $\epsilon\beta=0$, and, in the case (ii), by the additional relations $\lambda_i\delta_{i+1}=0$, $\delta_{i+1}\mu_{i}=0$, $\mu_i\lambda_i=0$.
\end{lemma}

\begin{proof}
Let $B$ be cluster-tilted of euclidean type. Because of \cite{ABS1}, there exists a tilted algebra $C$ such that $$B=\tilde C=C\ltimes\Ext^2(DC,C).$$
The hypothesis says that $C$ is not constricted. Because of Lemma \ref{lem 5.2}, $C$ is given by one of the bound quivers in (1) or (2) above. We examine these cases separately.

(1) Assume $C$ is given by the quiver
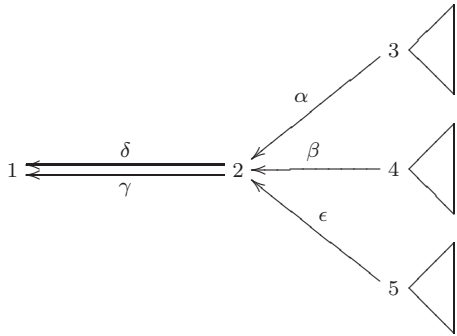
\[
\xy
\xymatrix"m"@C5pt@R5pt{{\scriptstyle3}\!\!&\!\! {\xy 
(0,0)*{}; (6,6)*{} **\dir{-};
(0,0)*{}; (6,-6)*{} **\dir{-};
(6,6)*{}; (6,-6)*{} **\dir{-};\endxy 
}\\
{\scriptstyle4}\!\!&\!\! {\xy 
(0,0)*{}; (6,6)*{} **\dir{-};
(0,0)*{}; (6,-6)*{} **\dir{-};
(6,6)*{}; (6,-6)*{} **\dir{-};\endxy 
}\\
{\scriptstyle5}\!\!&\!\! {\xy 
(0,0)*{}; (6,6)*{} **\dir{-};
(0,0)*{}; (6,-6)*{} **\dir{-};
(6,6)*{}; (6,-6)*{} **\dir{-};\endxy 
}
}
\POS-(20,16)
\xymatrix"n"{
   \scriptstyle 2
\ar@{<-}["m"]^\alpha  \ar@{<-}["m"d]^\beta  \ar@{<-}["m"dd]^\epsilon&
}
\POS-(30,0)
\xymatrix{\scriptstyle 1 
  \ar@<2pt>@{<-}["n"]^\delta\ar@<-2pt>@{<-}["n"]_\gamma
}
\endxy
\]
where the triangles are branches, possibly empty, bound by $\alpha \delta = 0$, $ \beta \delta = \beta \gamma$, $ \epsilon \gamma = 0$ and the branch relations.
Observe that, if one of the branches is empty, then it has no root and consequently, the arrow from that root to the point 2 does not exist.

We consider the following subcases:

(1a) Assume none of the branches rooted at 3,4,5 is empty. In this case, we refer to $C$ as $C_1$. Then the corresponding cluster-tilted algebra $B$ is of the form
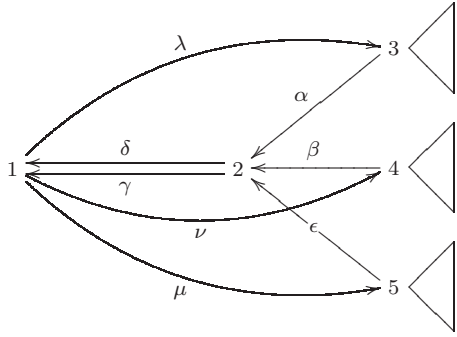
\[
\xy
\xymatrix"m"@C5pt@R5pt{{\scriptstyle3}\!\!&\!\! {\xy 
(0,0)*{}; (6,6)*{} **\dir{-};
(0,0)*{}; (6,-6)*{} **\dir{-};
(6,6)*{}; (6,-6)*{} **\dir{-};\endxy 
}\\
{\scriptstyle4}\!\!&\!\! {\xy 
(0,0)*{}; (6,6)*{} **\dir{-};
(0,0)*{}; (6,-6)*{} **\dir{-};
(6,6)*{}; (6,-6)*{} **\dir{-};\endxy 
}\\
{\scriptstyle5}\!\!&\!\! {\xy 
(0,0)*{}; (6,6)*{} **\dir{-};
(0,0)*{}; (6,-6)*{} **\dir{-};
(6,6)*{}; (6,-6)*{} **\dir{-};\endxy 
}
}
\POS-(20,16)
\xymatrix"n"{
   \scriptstyle 2
\ar@{<-}["m"]^\alpha  \ar@{<-}["m"d]^\beta  \ar@{<-}["m"dd]|{\,\epsilon\,}&
}
\POS-(30,0)
\xymatrix{\scriptstyle 1
  \ar@<2pt>@{<-}["n"]^\delta\ar@<-2pt>@{<-}["n"]_\gamma 
  \ar@/_20pt/["m"d]_\nu \ar@/^20pt/["m"]^\lambda \ar@/_20pt/["m"dd]_\mu
}
\endxy
\]
where the triangles are cluster-tilted algebras of type $\mathbb{A}$, and there are,  additionally, the relations of $C_1$ and the relations $\lambda\alpha=-\nu\beta$, $\nu \beta=\mu\epsilon$, $\delta\lambda=0$, $\delta\nu=\gamma\nu$ and $\gamma\mu=0$.

Now, this algebra $B$ can be written as $B$=$\tilde C_1'$, where $C_1'$ is given by the quiver
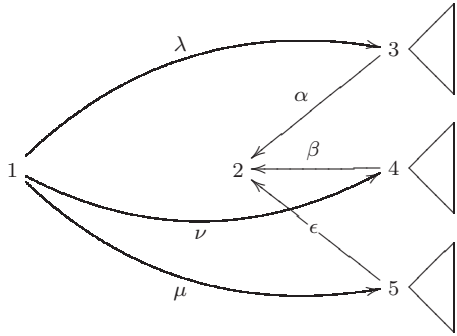
\[
\xy
\xymatrix"m"@C5pt@R5pt{{\scriptstyle3}\!\!&\!\! {\xy 
(0,0)*{}; (6,6)*{} **\dir{-};
(0,0)*{}; (6,-6)*{} **\dir{-};
(6,6)*{}; (6,-6)*{} **\dir{-};\endxy 
}\\
{\scriptstyle4}\!\!&\!\! {\xy 
(0,0)*{}; (6,6)*{} **\dir{-};
(0,0)*{}; (6,-6)*{} **\dir{-};
(6,6)*{}; (6,-6)*{} **\dir{-};\endxy 
}\\
{\scriptstyle5}\!\!&\!\! {\xy 
(0,0)*{}; (6,6)*{} **\dir{-};
(0,0)*{}; (6,-6)*{} **\dir{-};
(6,6)*{}; (6,-6)*{} **\dir{-};\endxy 
}
}
\POS-(20,16)
\xymatrix"n"{
   \scriptstyle 2
\ar@{<-}["m"]^\alpha  \ar@{<-}["m"d]^\beta  \ar@{<-}["m"dd]|{\,\epsilon\,}&
}
\POS-(30,0)
\xymatrix{\scriptstyle 1
  \ar@/_20pt/["m"d]_\nu \ar@/^20pt/["m"]^\lambda \ar@/_20pt/["m"dd]_\mu
}
\endxy
\]
where the triangles are again branches, bound by relations $\lambda\alpha=-\nu\beta$, $\nu \beta=\mu\epsilon$. 

This is easily seen to be a representation-finite tilted algebra of type $\tilde{\mathbb{D}}$ (indeed, one can simply construct the Auslander-Reiten quiver of the algebra and identify a complete slice). In particular, $C_1'$ is constricted, a contradiction.

(1b) Assume that the branch rooted,  say at 4, is empty while the other two are not. In this case,   we refer to $C$ as $C_2$. Then the cluster-tilted algebra $B$ is of the form

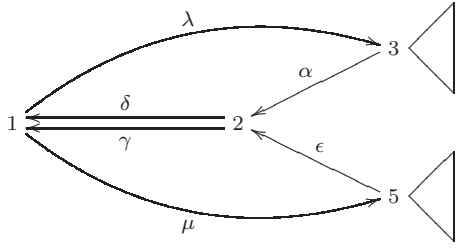
\[
\xy
\xymatrix"m"@C5pt@R5pt{{\scriptstyle3}\!\!&\!\! {\xy 
(0,0)*{}; (6,6)*{} **\dir{-};
(0,0)*{}; (6,-6)*{} **\dir{-};
(6,6)*{}; (6,-6)*{} **\dir{-};\endxy 
}\\ \\
{\scriptstyle5}\!\!&\!\! {\xy 
(0,0)*{}; (6,6)*{} **\dir{-};
(0,0)*{}; (6,-6)*{} **\dir{-};
(6,6)*{}; (6,-6)*{} **\dir{-};\endxy 
}
}
\POS-(20,10)
\xymatrix"n"{
   \scriptstyle 2
\ar@{<-}["m"]^\alpha   \ar@{<-}["m"dd]^\epsilon&
}
\POS-(30,0)
\xymatrix{\scriptstyle 1
  \ar@<2pt>@{<-}["n"]^\delta\ar@<-2pt>@{<-}["n"]_\gamma 
   \ar@/^20pt/["m"]^\lambda \ar@/_20pt/["m"dd]_\mu
}
\endxy
\]
where the triangles are cluster-tilted algebras of type $\mathbb{A}$, bound by the relations of $C_2$ and the additional relations $\lambda\alpha=0$, $\delta\lambda=0$,   $\gamma\mu=0$, $\mu\epsilon=0$.

Again, the algebra $B$ can be written as $B=\tilde C_2'$, where $C_2'$ is given by the quiver

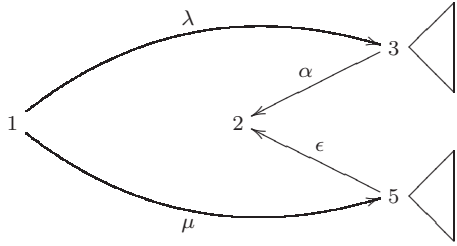
\[
\xy
\xymatrix"m"@C5pt@R5pt{{\scriptstyle3}\!\!&\!\! {\xy 
(0,0)*{}; (6,6)*{} **\dir{-};
(0,0)*{}; (6,-6)*{} **\dir{-};
(6,6)*{}; (6,-6)*{} **\dir{-};\endxy 
}\\ \\
{\scriptstyle5}\!\!&\!\! {\xy 
(0,0)*{}; (6,6)*{} **\dir{-};
(0,0)*{}; (6,-6)*{} **\dir{-};
(6,6)*{}; (6,-6)*{} **\dir{-};\endxy 
}
}
\POS-(20,10)
\xymatrix"n"{
   \scriptstyle 2
\ar@{<-}["m"]^\alpha   \ar@{<-}["m"dd]^\epsilon&
}
\POS-(30,0)
\xymatrix{\scriptstyle 1
   \ar@/^20pt/["m"]^\lambda \ar@/_20pt/["m"dd]_\mu
}
\endxy
\] 
where the triangles are branches,
bound by $\lambda\alpha=0, \mu\epsilon=0$ and the branch relations. This is easily seen to be a representation-finite tilted algebra of type $\tilde{\mathbb{A}}$ (see, for instance, \cite{AS}), thus $C_2'$ is constricted, another contradiction.

(1c) If at least two of the branches, say at 4 and 5, are empty, then  we are left with the  quiver (i) of the statement.

(2) Assume $C$ is given by the quiver
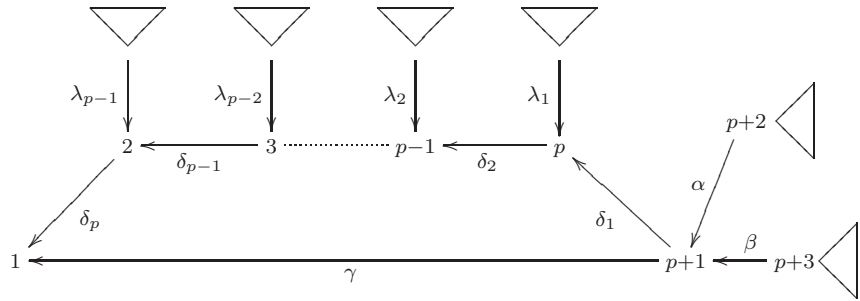
\[ \xy\xymatrix"m"@C5pt@R5pt{{\scriptstyle p+2}\!\!&\!\! {
\xy 
(0,0)*{}; (5,5)*{} **\dir{-};
(0,0)*{}; (5,-5)*{} **\dir{-};
(5,5)*{}; (5,-5)*{} **\dir{-};\endxy 
}} 
\POS-(96.5,-10)
\xymatrix@C20pt@R20pt{&\ar[d]_{\lambda_{p-1}}{\xy 
(0,0)*{}; (5,5)*{} **\dir{-};
(0,0)*{}; (-5,5)*{} **\dir{-};
(-5,5)*{}; (5,5)*{} **\dir{-};\endxy 
} &
{\xy 
(0,0)*{}; (5,5)*{} **\dir{-};
(0,0)*{}; (-5,5)*{} **\dir{-};
(-5,5)*{}; (5,5)*{} **\dir{-};\endxy 
}\ar[d]_{\lambda_{p-2}} &
{\xy 
(0,0)*{}; (5,5)*{} **\dir{-};
(0,0)*{}; (-5,5)*{} **\dir{-};
(-5,5)*{}; (5,5)*{} **\dir{-};\endxy 
}\ar[d]_{\lambda_{2}}
&
{\xy 
(0,0)*{}; (5,5)*{} **\dir{-};
(0,0)*{}; (-5,5)*{} **\dir{-};
(-5,5)*{}; (5,5)*{} **\dir{-};\endxy 
}\ar[d]_{\lambda_{1}}
&&~~
\\
 &  \scriptstyle 2\ar[dl]^{\delta_p} 
& \scriptstyle 3\ar[l]^{\delta_{p-1}}& \scriptstyle p-1\ar@{.}[l]  & \scriptstyle p\ar[l]^{\delta_2} \\
  \scriptstyle 1&   &  &  &  &   \scriptstyle p+1 \ar[lllll]^\gamma \ar[ul]^{\delta_1}
\ar@{<-}["m"uulllll]^\alpha&
{\scriptstyle p+3}\,{\xy 
(0,0)*{}; (5,5)*{} **\dir{-};
(0,0)*{}; (5,-5)*{} **\dir{-};
(5,5)*{}; (5,-5)*{} **\dir{-};\endxy 
}\ar[l]_\beta 
}
\endxy
\]where the   triangles are branches, possibly empty, bound by $\alpha \delta_1  \cdots \delta_p=\alpha \gamma$, $\beta \gamma = 0$, $\lambda_i \delta_{i+1}=0$ for all $1 \leq i < p$, and the branch relations.

We consider the following subcases. 

(2a) Assume that none of the branches rooted at $p+1, p+2$ is empty. In this case, we refer to $C$ as $C_3$. Then the corresponding cluster-tilted algebra is of the form
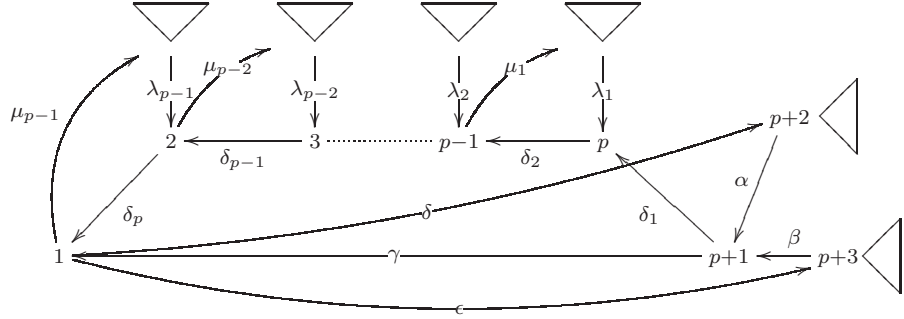
\[ \xy\xymatrix"m"@C5pt@R5pt{{\scriptstyle p+2}\!\!&\!\! {
\xy 
(0,0)*{}; (5,5)*{} **\dir{-};
(0,0)*{}; (5,-5)*{} **\dir{-};
(5,5)*{}; (5,-5)*{} **\dir{-};\endxy 
}} 
\POS-(96.5,-10)
\xymatrix@C20pt@R20pt{&\ar[d]|{\lambda_{p-1}}{\xy 
(0,0)*{}; (5,5)*{} **\dir{-};
(0,0)*{}; (-5,5)*{} **\dir{-};
(-5,5)*{}; (5,5)*{} **\dir{-};\endxy 
} &
{\xy 
(0,0)*{}; (5,5)*{} **\dir{-};
(0,0)*{}; (-5,5)*{} **\dir{-};
(-5,5)*{}; (5,5)*{} **\dir{-};\endxy 
}\ar[d]|{\lambda_{p-2}} &
{\xy 
(0,0)*{}; (5,5)*{} **\dir{-};
(0,0)*{}; (-5,5)*{} **\dir{-};
(-5,5)*{}; (5,5)*{} **\dir{-};\endxy 
}\ar[d]|{\lambda_{2}}
&
{\xy 
(0,0)*{}; (5,5)*{} **\dir{-};
(0,0)*{}; (-5,5)*{} **\dir{-};
(-5,5)*{}; (5,5)*{} **\dir{-};\endxy 
}\ar[d]|{\lambda_{1}}
&&~~
\\
 &  \scriptstyle 2\ar[dl]^{\delta_p} \ar@/^10pt/[ur]|{\mu_{p-2}}
& \scriptstyle 3\ar[l]^{\delta_{p-1}}& \scriptstyle p-1\ar@{.}[l]  \ar@/^10pt/[ur]|{\mu_{1}} & \scriptstyle p\ar[l]^{\delta_2} \\
  \scriptstyle 1\ar@/_20pt/[rrrrrr]|\epsilon\ar@/_10pt/["m"uu]|\delta\ar@/^20pt/[ruu]^{\mu_{p-1}}&   &  &  &  &   \scriptstyle p+1 \ar[lllll]|\gamma \ar[ul]^{\delta_1}
\ar@{<-}["m"uulllll]^\alpha&
{\scriptstyle p+3}\,{\xy 
(0,0)*{}; (5,5)*{} **\dir{-};
(0,0)*{}; (5,-5)*{} **\dir{-};
(5,5)*{}; (5,-5)*{} **\dir{-};\endxy 
}\ar[l]_\beta 
}
\endxy
\]
where the triangles are cluster-tilted algebras of type $\mathbb{A}$, bound by the relations of $C_3$ and the additional relations $\epsilon\beta=\delta\alpha$, $\gamma\epsilon=0$, $\delta_1\cdots\delta_p\delta=\gamma\delta$ and $\mu_i\lambda_i+\delta_{i+2}\cdots\delta_p\delta\alpha\delta_1\cdots\delta_i=0$, $\delta_{i+1}\mu_i=0$, 
for all $i$.

Now, this algebra can be written as $B=\tilde C_3'$, where $C_3'$ is given by the quiver
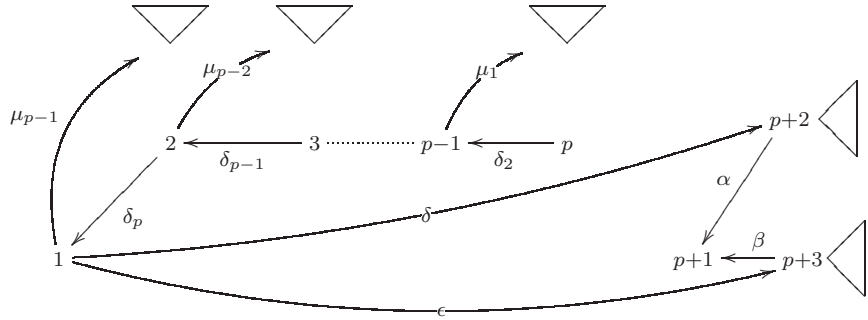
\[ \xy\xymatrix"m"@C5pt@R5pt{{\scriptstyle p+2}\!\!&\!\! {
\xy 
(0,0)*{}; (5,5)*{} **\dir{-};
(0,0)*{}; (5,-5)*{} **\dir{-};
(5,5)*{}; (5,-5)*{} **\dir{-};\endxy 
}} 
\POS-(96.5,-10)
\xymatrix@C20pt@R20pt{&{\xy 
(0,0)*{}; (5,5)*{} **\dir{-};
(0,0)*{}; (-5,5)*{} **\dir{-};
(-5,5)*{}; (5,5)*{} **\dir{-};\endxy 
}
 &
{\xy 
(0,0)*{}; (5,5)*{} **\dir{-};
(0,0)*{}; (-5,5)*{} **\dir{-};
(-5,5)*{}; (5,5)*{} **\dir{-};\endxy 
}
 &
&
{\xy 
(0,0)*{}; (5,5)*{} **\dir{-};
(0,0)*{}; (-5,5)*{} **\dir{-};
(-5,5)*{}; (5,5)*{} **\dir{-};\endxy 
}
&&~~
\\
 &  \scriptstyle 2\ar[dl]^{\delta_p} \ar@/^10pt/[ur]|{\mu_{p-2}}
& \scriptstyle 3\ar[l]^{\delta_{p-1}}& \scriptstyle p-1\ar@{.}[l]  \ar@/^10pt/[ur]|{\mu_{1}} & \scriptstyle p\ar[l]^{\delta_2} \\
  \scriptstyle 1\ar@/_20pt/[rrrrrr]|\epsilon\ar@/_10pt/["m"uu]|\delta\ar@/^20pt/[ruu]^{\mu_{p-1}}&   &  &  &  &   \scriptstyle p+1
\ar@{<-}["m"uulllll]^\alpha&
{\scriptstyle p+3}\,{\xy 
(0,0)*{}; (5,5)*{} **\dir{-};
(0,0)*{}; (5,-5)*{} **\dir{-};
(5,5)*{}; (5,-5)*{} **\dir{-};\endxy 
}\ar[l]_\beta 
}
\endxy
\]
with the inherited relations. This is again seen to be a representation-finite tilted algebra of type $\tilde \mathbb{D}$. In particular, it is constricted,  a contradiction.

(2b) If at least one of the branches, say at $p+2$ is empty, then we are left with the quiver (ii) of the statement.
\qed
\end{proof}

\bigskip

Observe that in the proof of Lemma \ref{lem 5.3}, in each of the cases (1a), (1b) and (2a), we have replaced the original non-constricted tilted algebra $C_1,C_2$ and $C_3$ by a constricted one $C_1',C_2'$ and $C_3'$, respectively.

\begin{lemma}
 \label{lem 5.*}
 With the above notation, for each $i\in\{1,2,3\}$, we have $n_{B,C_i}=n_{B,C_i'}$ and ${\HH}^1(C_i)\cong {\HH}^1(C_i')$.
\end{lemma}
\begin{proof}
The first statement follows immediately from the description of the relations in the respective algebras. Thus $n_{B,C_1}=n_{B,C_1'}=1$, $n_{B,C_2}=n_{B,C_2'}=2$ and $n_{B,C_3}=n_{B,C_3'}=1.$

It suffices to show the second statement. We consider each of the cases as in the proof of Lemma \ref{lem 5.3}.

(1a) Let $D_1$ be the full convex subcategory of $C_1$ (and $C_1'$) generated by all points except the point 1. Then $D_1$ is a representation-finite tilted algebra and $C_1$ (or $C_1'$) is a one-point coextension (or extension, respectively) of $D_1$ by an indecomposable module. This module being a rigid brick, we deduce immediately from Happel's sequence \cite[5.3]{Ha2} that
\[{\HH}^1(C_1)\cong\HH^1(D_1) \cong {\HH}^1(C_1').\]

(1b) Let $D_2$ be the full convex subcategory of $C_2$ (and $C_2'$) generated by all points except the point 1. Then $D_2$ is a representation-finite tilted algebra and $C_2$ (or $C_2'$) is a one-point coextension (or extension, respectively) of $D_2$ by the direct sum of two Hom-orthogonal, rigid bricks $X,Y$ such that $\Ext^1_{D_2}(X,Y)=0$ and $\Ext^1_{D_2}(Y,X)=0$. Again Happel's sequence yields
\[{\HH}^1(C_2)\cong {\HH}^1(D_2)\cong{\HH}^1(C_2').\]

(2a) Let $D_3$ be the full convex subcategory of $C_3$ (and $C_3'$) generated by all points except the points $1,2,\cdots,p$. 
Then there is a sequence 
\[C_3=E_0\supsetneq E_1\supsetneq \cdots \supsetneq E_p=D_3,\]
where $E_i$ is a one-point coextension of $E_{i+1}$. Moreover, each $E_i$ is a direct product of representation-finite tilted algebras and the coextension module is a direct sum of rigid bricks with supports in distinct connected components of $E_i$. Similarly, there is a sequence 
\[C_3'=F_p\supsetneq F_{p-1}\supsetneq \cdots \supsetneq F_0=D_3,\]
where $F_{i+1}$ is a one-point extension of $F_{i}$. Moreover, each $F_i$ is a direct product of representation-finite tilted algebras and the extension module is a direct sum of rigid bricks with supports in distinct connected components of $F_i$.
Therefore easy inductions yield
\[{\HH}^1(C_3)\cong{\HH}^1(D_3)\cong{\HH}^1(C_3').\]
\qed
\end{proof}

\begin{lemma}
 \label{lem 5.4}
 Let $B=\tilde C$ be a non-hereditary cluster-tilted algebra of type $\tilde\mathbb{A}$ of one of the forms of Lemma \ref{lem 5.3} and  $R$ a system of relations for $C$. Then
 \begin{itemize}
\item [\textup{(i)}] If $B$ is of the form (i), then ${\HH}^1(B) = k^{|R|+2}$
\item [\textup{(ii)}] If $B$ is of the form (ii), then ${\HH}^1(B) = k^{|R|+1}$
\end{itemize}
\end{lemma}
\begin{proof}
 (i) We use the formula of \cite{CS}, as applied to our special situation in \cite[Proposition 5.1]{AR}
 \[\dim {\HH}^1(B)=\dim Z(B)-|\tilde Q_0\, /\!/\,N| +|\tilde Q_1\, /\!/\,N| -|(\tilde Q_1\, /\!/\,N)_e| -\dim \Im R_g. 
 \]
 Here, $Z(B)$ is the centre of $B$, so $\dim Z(B)=1$. Next, $\tilde Q_0\, /\!/\,N$ is the set of non-zero oriented cycles in $(\tilde Q,\tilde I)$ (where, as usual, $B=k\tilde Q/\tilde I$) including points. Then 
 \[|\tilde Q_0\, /\!/\,N| =|\tilde Q_0|= |Q_0|.\]
 Thirdly, $\tilde Q_1\, /\!/\,N$ is the set of pairs $(\alpha,w)$, where $\alpha\in \tilde Q_1$ and $w$ is a non-zero path (of length $\ge 0$) parallel to $\alpha$. This consists of all pairs $(\alpha,\alpha)$, with $\alpha\in \tilde Q_1$ and the two pairs $(\delta,\gamma)$, $(\gamma, \delta)$ arising from the double arrow
 \[\xymatrix@C60pt{1&2 \ar@<2pt>[l]^\gamma\ar@<-2pt>[l]_\delta}.\]
 Thus, $|\tilde Q_1\, /\!/\,N| =|\tilde Q_1|+2.$
 
 Since it is shown in \cite[Proof of Proposition 5.1]{AR} that $R_g=0$, there remains to compute 
 \[ (\tilde Q_1\, /\!/\,N)_e = (\tilde Q_1\, /\!/\,N)\setminus \left((\tilde Q_1\, /\!/\,N)_g \cup (\tilde Q_1\, /\!/\,N)_a\right).\]
 Here:
 \begin{enumerate}
\item $(\tilde Q_1\, /\!/\,N)_g $ is the set of all pairs $(\alpha,w)\in \tilde Q_1\, /\!/\,N $ where $w$ is either a point or a path starting or ending with the arrow $\alpha$. Therefore
\[(\tilde Q_1\, /\!/\,N)_g =\{(\alpha,\alpha)\mid\alpha\in\tilde Q_1\}.\]
 \item  $(\tilde Q_1\, /\!/\,N)_a $ is the set of all pairs $(\alpha,w)\in \tilde Q_1\, /\!/\,N $ where, in each relation where $\alpha $ appears, replacing $\alpha$ by $w$ yields a zero path. Therefore
\[(\tilde Q_1\, /\!/\,N)_a =\{(\alpha,\alpha)\mid\alpha\in\tilde Q_1\}\cup \{(\delta,\gamma)\}.\]
\end{enumerate}
This implies that 
\[|(\tilde Q_1\, /\!/\,N) |
-|(\tilde Q_1\, /\!/\,N)_e | =
|(\tilde Q_1\, /\!/\,N)_g
\cup (\tilde Q_1\, /\!/\,N)_a |
=|\tilde Q_1|+1.\]
Therefore 
\begin{eqnarray*}
 \dim {\HH}^1(B)&=&1-|\tilde Q_0|+|\tilde Q_1 | +1 \\
 &=&1-| Q_0|+| Q_1 | +|R| +1 \\
 &=& |R|+2,
\end{eqnarray*}
because $|\tilde Q_1|=|Q_1|+|R|$ and $|Q_0|=|Q_1|.$

(ii) For this case again $\dim Z(B)=1$ and $|(\tilde Q_0\, /\!/\,N) | =|\tilde Q_0 |=|Q_0|$.
Here
\[ \tilde Q_1\, /\!/\,N = \{(\alpha,\alpha) \mid\alpha\in\tilde Q\}\cup \{(\gamma,\delta_1\cdots\delta_p)\}.\]
Now, as before
\[(\tilde Q_1\, /\!/\,N)_g =\{(\alpha,\alpha)\mid\alpha\in\tilde Q_1\},\]
while 
\[(\tilde Q_1\, /\!/\,N)_a =\{(\alpha,\alpha)\mid\alpha\in\tilde Q_1\},\]
so that 
\[|(\tilde Q_1\, /\!/\,N) |
-|(\tilde Q_1\, /\!/\,N)_e | =
|(\tilde Q_1\, /\!/\,N)_g
\cup (\tilde Q_1\, /\!/\,N)_a |
=|\tilde Q_1|.\]
Therefore 
\begin{eqnarray*}
 \dim {\HH}^1(B)&=&1-|\tilde Q_0|+|\tilde Q_1 | \\
 &=&1-| Q_0|+| Q_1 | +|R| \\
 &=& |R|+1,
\end{eqnarray*}
because $|\tilde Q_1|=|Q_1|+|R|$ and $|Q_0|=|Q_1|.$
\qed
\end{proof}

\medskip
Note that the previous proof can also be done applying the formula in \cite{RR}.

\section{Hochschild cohomology of  tame cluster-tilted algebras}

We need a few results from \cite{AR} and \cite{ABIS} which we now recall.  Let $B$ be a split extension of a subalgebra $C$ by a two-sided bimodule $_C E_C$, that is,
{ let $B$ have the $k$-vector space structure of $C \oplus E$ with the multiplication given by
$$(c,x)(c',x')= (cc', cx'+xc'+xx')$$
for $(c,x),(c',x') \in B$.  Then there exists a short exact sequence of $C-C$-bimodules
\[ \xymatrix{ 0 \ar[r] & E \ar[r] & B \ar[r]^p  & C \ar[r] & 0 }\]
where $p: (c,x) \mapsto c$ is an algebra morphism, and there is a morphism $q: c \mapsto (c,0)$ of $C-C$- bimodules (but also of algebras) such that $pq=id_C$.} Thus, in particular, a trivial extension $B= C \ltimes E$ is a split extension such that $E^2=0$.  We recall the following result from \cite{AR}.

\begin{lemma}\cite[Lemma 4.1]{AR} If $B$ is a split extension of $C$, then there exists a morphism $\varphi: {\HH}^1(B) \to {\HH}^1(C)$ given by $[\delta] \mapsto [p\delta q]$.
\end{lemma}

The morphism $\varphi$ is shown in \cite{ABIS} to be surjective in case $C$ is a tilted algebra and $B$ is its relation-extension.  In fact we have

\begin{theorem}\cite[Theorem 3.5]{ABIS} \label{42} Let $B$ be  the trivial extension of a tilted algebra $C$ by the relation bimodule $E= \Ext^2_C(DC,C)$, then there exists a short exact sequence of vector spaces
\[ \xymatrix{ 0 \ar[r] &  {\HH}^1(B,E)  \ar[r] &  {\HH}^1(B)   \ar[r]^\varphi  & {\HH}^1(C) \ar[r] & 0. }\]
\end{theorem}

In the sequel we always write $E= \Ext^2_C(DC,C)$.  Our objective in this section is to prove the following theorem.

\begin{theorem}\label{31}
Let $B$ be a  tame cluster-tilted algebra, and $C$ a tilted algebra such that $B=C \ltimes E$.  Then there exists a short exact sequence of $k$-vector spaces
\[ \xymatrix { 0 \ar[r] &  k^{n_{B,C}} \ar[r] & {\HH}^1(B) \ar[r]^\varphi &  {\HH}^1(C)  \ar[r] & 0.} \] 
\end{theorem}

Applying Theorem \ref{42}, it suffices to prove that ${\HH}^1(B,E) = k^{n_{B,C}} $.  We set $n=n_{B,C}$ for simplicity.  Our first task is to prove that ${\HH}^1(B,E)$ can be written in a simpler form.  This is achieved in the following two statements.

\begin{lemma} \label{directa}
Let $B = C \ltimes E$ with $C$ a tilted algebra, then $\Der_0(B,E)= \Der_0(C,E) \oplus \End_{C^e} E$.
\end{lemma}

\begin{proof}  
Let $\delta \in \Der_0(B,E)$, then we can define two $k$-linear maps $d: C \to E$ and $f:E \to E$ by
$$d(c) = \delta (c,0)  \quad \mbox{for all $c \in C$},$$
$$f(x) = \delta (0,x)  \quad \mbox{for all $x \in E$},$$
that is, $d= \delta \vert_C$ and $f=\delta \vert_E$.

We first prove that $d: C \to E$ is a normalized derivation.  Indeed, let $c, c' \in C$ then
\[ \begin{array}{lll}
d(c c') &=& \delta (c c' , 0) = \delta ((c,0) (c',0)) \\ 
&=& (c,0) \delta (c',0) + \delta (c,0) (c',0) \\
&=& c d(c') + d(c) c'.
\end{array}\]
On the other hand, $d(e_i)= \delta (e_i,0) =0$ for every $i$.

Next, we prove that $f:E \to E$ is a morphism of $C-C$-bimodules.  Let $c \in C$ and $x \in E$, then 
\[ \begin{array}{lll}
f(cx)&=& \delta(0,cx) = \delta ((c,0)(0,x))\\ 
&=& (c,0) \delta(0,x) + \delta (c,0) (0,x) \\
&=&c f(x) + d(c) x = c f(x)
\end{array}\]
because $d(c), x \in E$ and $E^2=0$.  Similarly, $f(xc)=f(x)c$.

We have shown that $\Der_0(B,E)=\Der_0(C,E)+\End_{C^e} E$.  But now $\Der_0(C,E) \subseteq \Hom_k(C,E)$ while $\End_{C^e} E \subseteq \Hom_k(E,E)$ and we have an obvious direct sum decomposition 
$$\Hom_k(B,E) = \Hom_k(C,E) \oplus \Hom_k (E,E).$$
Therefore $\Der_0(B,E)= \Der_0(C,E) \oplus \End_{C^e} E$. 
\qed 
\end{proof}

\medskip

\begin{proposition}
Let $B=C \ltimes E$ with $C$ a tilted algebra, then
${\HH}^1(B,E) = {\HH}^1(C,E) \oplus \End_{C^e} E$.
\end{proposition}

\begin{proof}
Because of Lemma \ref{directa}, we have a direct sum decomposition $\Der_0(B,E)= \Der_0(C,E) \oplus \End_{C^e} E$.  We prove that it induces a direct sum decomposition $\Int_0(B,E)= \Int_0(C,E) \oplus 0$ on the level of the inner derivations, that is, if $\delta \in \Int_0(B,E)$ then $d=\delta \vert_C \in \Int_0(C,E)$ while $f=\delta \vert_E =0$.

Assume $\delta \in \Int_0(B,E)$ then there exists $(c,x) \in B=C \oplus E$ such that $\delta = \delta_{(c,x)}$, that is, for all $(c',x') \in B$, we have 
\[ \begin{array}{lll}
\delta (c',x') &=& (c,x) (c',x') - (c',x') (c,x) \\ 
&=& (c c', xc' + c x') - (c' c, x' c + c' x)\\
&=& (c c' - c' c, c x' - x' c + x c' - c' x).
\end{array}\]
But $\delta (c', x') \in E$.  Therefore $c c' - c' c =0$, or $c c' = c' c$.  Since $c'$ is an arbitrary element of $C$, this means that $c$ is in the center of $C$. Now $C$ is a tilted algebra, so it is triangular and therefore $c= \lambda$ is a scalar.  But then
$$\delta (c', x') = xc' - c' x \in E.$$
Now, let $d= \delta \vert_C$ and $f=\delta \vert_E$.  Then, for all $x' \in E$, we have
$$f(x') = \delta (0, x') =0$$
so that $f=0$, as desired.  On the other hand, $d(c') = \delta (c',0) = x c' - c' x = [ x, c' ]$, that is, $d = [x, -] \in \Int_0(C,E)$.  This establishes our claim. But then we have
\[ \begin{array}{lll}
{\HH}^1(B,E) &=& \frac{\Der_0 (B,E)} {\Int_0(B,E)} \simeq  \frac {\Der_0(C,E) \oplus \End_{C^e} E} {\Int_0(C,E) \oplus 0} \\ 
&\simeq &   \frac {\Der_0(C,E)} {\Int_0(C,E)}  \oplus  \End_{C^e} E   \\
&=& {\HH}^1(C,E) \oplus \End_{C^e} E.
\end{array}\]
\qed
\end{proof}

\medskip

We shall now prove that if $C$ is constricted then ${\HH}^1(C,E)=0$.  Our proof will use essentially the tameness of $B$, which implies that $C$ is a tilted algebra of Dynkin or euclidean type.

\begin{lemma} \label{4.6}
Assume that $C$ is tilted and constricted, and that $B = C \ltimes E$ is tame. Then ${\HH}^1(C,E)=0$.
\end{lemma}

\begin{proof}
It suffices to prove that $\Der_0(C,E)=0$.  Assume thus that $\alpha: i \to j$ is an old arrow, then $\delta \in\Der_0(C,E)$ implies $\delta (\alpha) = e_i \delta (\alpha) e_j \in e_i E e_j$.
Now, assume $w \in E$ is non-zero, then $e_i w e_j$ contains a new arrow $\beta$ and we have the following situation
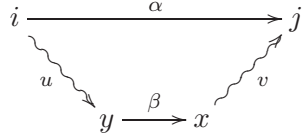
\[ \xymatrix{
i\ar@{~>}[dr]_u\ar[rrr]^\alpha&&&j \\
 & y \ar[r]^\beta& x \ar@{~>}[ur]_v }
\]
with $\beta$ a new arrow and $u, \alpha, v$ old. Let $\rho= \sum_i \lambda_i w_i$ be the relation in $C$ corresponding to $\beta$. Because $C$ is triangular, at least one of the paths $u$ or $v$ is nontrivial.  We have  four cases to consider.
\begin{itemize}
\item[(i)] If the paths $u$ and $v$ share no arrows with any of the $w_i$ then the bound quiver of the full subcategory of $C$ generated by $i, j$ and all the points lying on the summands of $\rho$, contains a sequential walk $w_1  u^{-1} \alpha v^{-1} w_1$, a contradiction to Lemma \ref{2.4}. 

\item [(ii)] The paths $u, v$ cannot be in the situation of intersecting one of the $w_i$ but sharing neither the last arrow of $u$ nor  the first arrow of $v$.  Indeed, assume that this happens, say, with $u$ (the case where it happens with $v$ is treated similarly).  Then we have $u= u' \gamma u''$ with $\gamma$ a subpath of $w_1$, that is, $w_1 = w'  \gamma  w''$
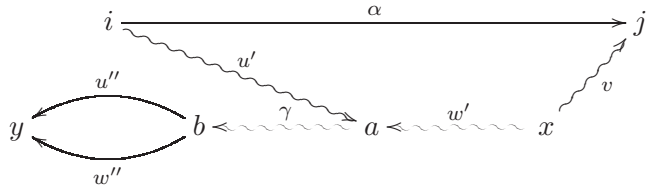
\[ \xymatrix{ 
& i     \ar@{~>}[drrr]^{u'}  \ar[rrrrrr]^{\alpha}  & &&&&  &  j \\
y &  &  \ar@/^1pc/[ll]^{w''}  \ar@/_1pc/[ll]_{u''}  b & &     a     \ar@{~>}[ll]_\gamma     &  & x  \ar@{~>}[ll]_{w'}     \ar@{~>}[ur]_v }\] 
Then $\rho$ is monomial because, otherwise, letting $c$ be any point of $w_2$ distinct from $x,y$, we have a wild full subcategory of $C$
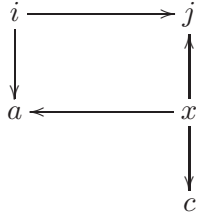
\[ \xymatrix{ i  \ar[rr] \ar[d] & & j \\
a && x \ar[u] \ar[d] \ar[ll]  \\
&& c  } \]
a contradiction.

Next, $u \not = 0$ implies $u'' \not = 0$ and the fact that $\rho$ is a relation implies $w'' \not = 0$.  But $C$ cannot contain two tame concealed full subcategories (namely, corresponding to the cycles $(u'')^{-1} w''$ and $u' (w')^{-1} v \alpha^{-1}$).  Therefore, there is a relation linking $u''$ and $w''$.  In particular, both have length at least two.  Because of \cite[4.7]{AS2} we must have $b=a$.  On the other hand, we have $a \not = x$, because otherwise $\alpha$ is parallel to $u'v$ which contradicts our assumption that $C$ is constricted.  This shows that $C$ contains a full subcategory $D$ of the form 
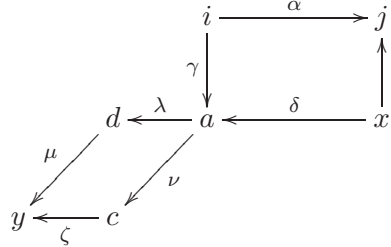
\[  \xymatrix{  & & i  \ar[rr]^\alpha  \ar[d]_\gamma  && j \\
& d \ar[ld]_\mu & a \ar[ld]^\nu  \ar[l]_\lambda && x \ar[ll]_\delta \ar[u] \\
y & c \ar[l]^\zeta } \]
where we have $\gamma \lambda \mu \not =0$ and $\delta \nu \zeta =0$ (thus $\delta \nu \not =0, \nu \zeta \not =0$) and there is a relation linking $\lambda \mu$ and $\nu \zeta$.  Because $D$ contains no wild full subcategory and must satisfy  \cite[4.7]{AS2}
 we must have $\delta \lambda = 0$ and $\gamma \nu =0$.   Let $D'$ be the full convex subcategory of $D$ consisting of all points except $y$.  Then $D'$ is the one-point coextension  of a tame hereditary algebra of type $\tilde \A_{2,2}$ by two indecomposable regular modules lying on two different tubes of rank $2$.  Therefore, applying \cite[4.7]{Ri}, we get that $D'$ is a tilted algebra of type $\tilde \A_{3,3}$ having a complete slice in the postprojective component.  Now $D$ is the one-point coextension of $D'$ by an indecomposable $D'$-module $M$ such that $\Hom_{D'} (M, I_d) \not = 0$ and $\Hom_{D'} (M, I_c) \not = 0$.  Then $M$ maps nontrivially into two different tubes of $D'$ and so $M$ is postprojective.  Therefore $D$ is a tilted algebra of wild type
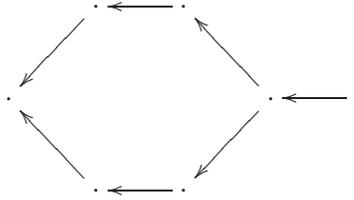
\[ \xymatrix{  & \cdot \ar[ld]  & \cdot \ar[l]  \\
\cdot & && \cdot \ar[lu] \ar[ld] & \ar[l]  \cdot \\
& \cdot \ar[lu]  & \cdot  \ar[l] } \] 
a contradiction.

\item[(iii)]  If the paths $u$ and $v$ share the last arrow of $u$ and the first arrow of $v$ with two different $w_i$, say $w_1$ and $w_2$, then, in particular, $\rho$ is not a monomial relation.  We then have the following situation
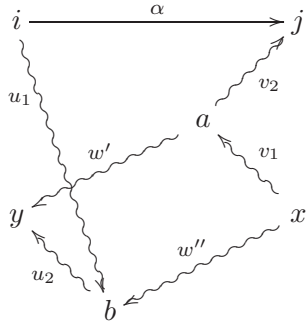
\[ \xymatrix{
i\ar@{~>}[dddr]_(.25){u_1}\ar[rrr]^\alpha&&&j \\
 & &   a \ar@{~>}[ur]_{v_2} \ar@{~>}[dll]_{w'} \\
y &&& x  \ar@{~>}[ul]_{v_1} \ar@{~>}[dll]_{w''} \\
& b \ar@{~>}[ul]^{u_2}
}
\]
with $u=u_1 u_2, v=v_1 v_2, w_1= w'' u_2$ and $w_2 = v_1 w'$.  Assume first that $\rho$ is not a binomial relation.  Then $C$ contains a full subcategory $C'$ of the form
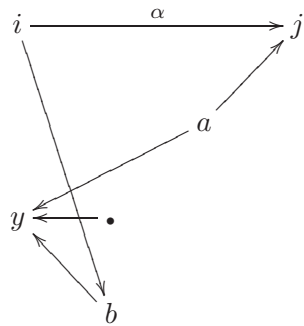
\[ \xymatrix{
i\ar[dddr]\ar[rrr]^\alpha&&&j \\
 & &   a \ar[ur] \ar[dll] \\
y & \scriptscriptstyle\bullet \ar[l] && \\
& b \ar[ul]
}
\]
and, since $C'$ is wild, we get a contradiction to the tameness of $B$.  Therefore $\rho$ is binomial.  Consider the full subcategory $H$ of $C$ generated by the points $i,j,a,y,b$. Then $H$ is hereditary of type $\tilde \A$ and $C$ contains a full subcategory $C'=H[M]$, where $M$ is the indecomposable $H$-module with semisimple socle  $S_y \oplus S_j$ (or, $S_y \oplus S_j \oplus S_j$ if $u_1$ is trivial) and such that $M/\soc(M) = S_a \oplus S_b$.  Then $M \cong (P_a \oplus P_b)/P_y \cong \tau^{-1} P_y$ is postprojective, and hence $C'=H[M]$ is a tilted algebra of wild type, a contradiction. 

\item[(iv)] If the paths $u$ and $v$ share the last arrow of $u$ and the first arrow of $v$ with the same $w_i$, say $w_1$, then we have the following situation
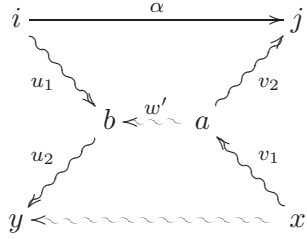
\[ \xymatrix{
i\ar@{~>}[dr]_{u_1}\ar[rrr]^\alpha&&&j  \\
& b \ar@{~>}[dl]_{u_2} & a \ar@{~>}[l]_{w'}  \ar@{~>}[ur]_{v_2}\\
 y  & && x  \ar@{~>}[lll]\ar@{~>}[ul]_{v_1}  }
\]
with $u=u_1 u_2, v=v_1v_2, w=v_1 w' u_2$. If both $u_2$ and $v_1$ are non-trivial, then $C$ contains a wild full subcategory.  Thus $u_2$ or $v_1$ is trivial.  Assume that $u_2$ is trivial, then $y=b$ and $u=u_1$ (the case where $v_1$ is trivial is entirely similar).  If $\rho$ is not a monomial relation, or the path $v_1$ contains more than one arrow, then $C$ contains a wild full subcategory.  Therefore, $\rho$ is monomial and $v_1$ consists of one arrow $\gamma$.  We have the following situation
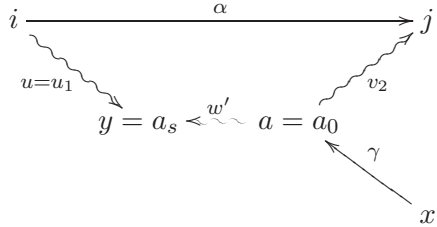
\[ \xymatrix{
i\ar@{~>}[dr]_{u=u_1}\ar[rrr]^\alpha&&&j  \\
& y = a_s   & a = a_0  \ar@{~>}[l]_{w'}  \ar@{~>}[ur]_{v_2}\\
  & && x  \ar[ul]_\gamma  }
\]
Let $H$ denote the full subcategory of $C$ generated by the points $i, j$, $a=a_0$, $a_1, \dots, a_s=y$.  Then $H$ is hereditary of type $\tilde \A$, and $C$ contains a full subcategory $C' = H[M]$, where $M$ is the $H$-module with top $S_a$ and socle $S_{a_{s-1}} \oplus S_j$ (or, $S_{a_{s-1}}$ if $v_2$ is trivial). Then $M$ is regular lying in an exceptional tube.  In this case, $C'=H[M]$ is wild, except for $s=1$ in which case it is tilted of type $\tilde \A$,  see \cite{AS}.  Then we have the following situation
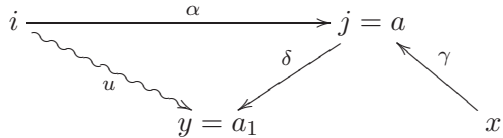
\[ \xymatrix{
i\ar@{~>}[drr]_{u}\ar[rrr]^\alpha&&&j = a \ar[dl]_\delta \\
&&  y = a_1   & &  x  \ar[ul]_\gamma  }
\]
In particular, the potential has a summand of the form $w_1= \gamma \delta \beta$, hence $B$ has a relation of the form $\beta \gamma =0$.  But then $e_i w e_j =0$ as required. 
\end{itemize}

\qed
\end{proof} 

\medskip

We next study the $C-C$-bimodule endomorphisms of $E$.  The following notation will be useful.  If $u$ is a subpath of a path $v$ we shall say that $u$ {\it divides} $v$ and write $u \vert v$.  In particular, if an arrow $\alpha$ (or a point $x$) lies on the path $v$, then we write $\alpha \vert v$ (or $x \vert v$, respectively). We use the symbol $\not \vert$ in the obvious way.

{  We recall from \cite {ABIS} that, if $\S_1, \cdots, \S_n$ are the distinct equivalence classes of new arrows, and $E_j$ is the $C-C$-bimodule generated by the elements of $\S_j$, then we have $E = \oplus_{j=1}^n E_j$ (see \cite[4.3]{ABIS}).  }

\begin{lemma}\label{45}
Let $C$ be  constricted and $E = \oplus_{i=1}^n E_i$ be the decomposition induced from the arrow equivalence relation then, for every new arrow $\alpha$ in $E_i$ and every $\delta \in \Hom_{C^e}(E_i, E)$, we have
$$\delta (\alpha) = \lambda_\alpha \alpha$$
for some scalar $\lambda_\alpha \in k$.
\end{lemma}

\begin{proof}
Let us denote by $\{ \alpha_1, \cdots, \alpha_n \}$ the set of all new arrows and by $\{ \rho_{\alpha_1}, \cdots, \rho_{\alpha_n} \}$ the set of corresponding relations in $C$ so that the potential is
$$W= \sum_{i=1}^n \rho_{\alpha_i} \alpha_i .$$
We may assume that the equivalence class  $[ \alpha_1]$ of $\alpha_1$ is $\{ \alpha_1, \cdots, \alpha_r \}$ with $r \leq n$.
For any $i$ with $1 \leq i \leq r$ and $\delta$ as in the statement, we clearly have
$$\delta (\alpha_i) = \sum_j \lambda_{ij} u_{ij} \alpha_j v_{ij}$$
where $\lambda_{ij} \in k$, and $u_{ij}, v_{ij}$ are paths in $C$ such that $\alpha_i$ and $u_{ij} \alpha_j v_{ij}$ are parallel.

We claim that $\lambda_{ik} = 0$ when $k \not = i$, and this implies  $\delta (\alpha_i)= \lambda_{ii} \alpha_i$ because $C$ is triangular.  Without loss of generality, let $i=1$ and assume that $k \not =1$, then $\rho_{\alpha_k} \not = \rho_{\alpha_1}$ and hence there exists an arrow $\gamma_k$ such that $\gamma_k \vert \rho_{\alpha_1}$ but $\gamma_k \! \! \not \vert \rho_{\alpha_k}$.  Deriving the potential $W$ yields
$$\partial_{\gamma_k}(W) = \sum_{s \in S, t \in T, t\not =k} w_{st} \alpha_t w'_{st} \in \tilde I$$
that is, the arrow $\alpha_k$ does not appear in the above sum.  There exists a summand of $\partial_{\gamma_k}(W)$ which is a minimal relation in $\tilde I$ which contains $\alpha_1$ but not $\alpha_k$. Further, because $[\alpha_1] = \{ \alpha_1, \cdots, \alpha_r\}$, those $\alpha_t$ which appear in this minimal relation are such that $t \leq r$.  Let thus this minimal relation be
$$\rho = \sum_{s \in S', t \in T', t\not =k, 1 \in T'} w_{st} \alpha_t w'_{st}.$$
Applying our morphism $\delta: E_i \to E$ yields
$$0=\delta (\rho) = \sum_{s,t,j} \lambda_{tj} w_{st} u_{tj} \alpha_j v_{tj} w'_{st}.$$
The summands for which $t=1, j=k$ are of the form
$$ \lambda_{1k} w_{s1} u_{1k} \alpha_k v_{1k} w'_{s1}.$$
The arrow $\alpha_k$ belongs to the cycle $\alpha_k\rho_{\alpha_k}$ in the potential $W$ and therefore the above summands contain each a subpath of the form
$u \alpha_k v$, 
where $v \epsilon u$ (for some arrow $\epsilon: x \to y$) is a subpath dividing $\rho_{\alpha_k}$.  We split the proof into several cases.
\begin{itemize}
\item [(i)] $x \vert v_{1k}$ and $y \vert u_{1k}$. Then we have the following situation
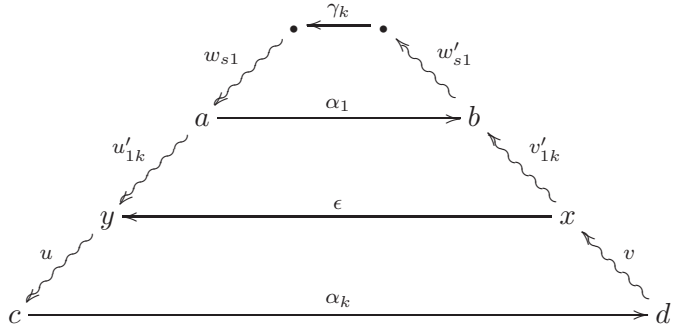
\[ \xymatrix{ &&&\scriptscriptstyle\bullet\ar@{~>}[dl]_{w_{s1}}&\scriptscriptstyle\bullet\ar[l]_{\gamma_k}\\
&&a \ar@{~>}[dl]_{u'_{1k}} \ar[rrr]^{\alpha_1}&&&b  \ar@{~>}[ul]_{w'_{s1}} \\
&y \ar@{~>}[dl]_{u} &&&&& x  \ar@{~>}[ul]_{v'_{1k}} \ar[lllll]_\epsilon  \\
c \ar[rrrrrrr]^{\alpha_k}&&&&&&& d \ar@{~>}[ul]_{v}
}\]
where $u_{1k} = u'_{1k}u$ and $v_{1k} = v v'_{1k}$.  In this case, we have the relation $\rho_{\alpha_1}$ which involves the path $w'_{s1} \gamma_k w_{s1}$ and perhaps other paths from $b$ to $a$ in $C$.   Now, the paths $v'_{1k}, u'_{1k}$ cannot both be trivial, because otherwise $B$ would contain a $2$-cycle formed by the arrows $\alpha_1$ and $\epsilon$. On the other hand, $v'_{1k}$ and $u'_{1k}$ do not share arrows with the summands of $\rho_{\alpha_1}$ (because of their directions). We consider the full subcategory of $C$ generated by all the points lying on the summands of $\rho_{\alpha_1}$  and the points $x,y$.   We thus get a contradiction to Lemma \ref{2.4} since $(w'_{s1} \gamma_k w_{s1}) u'_{1k} \epsilon^{-1} v'_{1k} (w'_{s1} \gamma_k w_{s1})$ is a sequential walk. 

\item [(ii)] $x \vert w'_{s1}$ and $y \vert w_{s1}$. Then we have the following situation

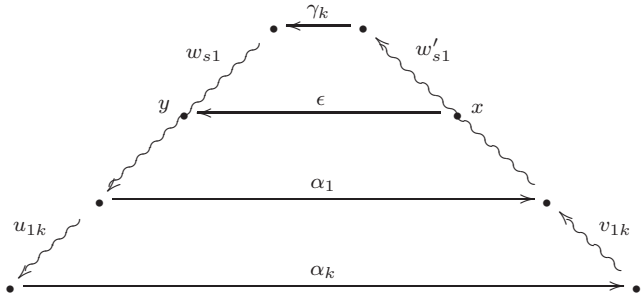
\[ \xymatrix{ &&&\scriptscriptstyle\bullet\ar@{~>}[ddll]_(.25){w_{s1}}_(.55)y&\scriptscriptstyle\bullet\ar[l]_{\gamma_k}\\
&&\!\scriptscriptstyle\bullet&&&\,\scriptscriptstyle\bullet\ar[lll]_\epsilon \\
&\scriptscriptstyle\bullet\ar@{~>}[dl]_{u_{1k}} \ar[rrrrr]^{\alpha_1}&&&&&\scriptscriptstyle\bullet \ar@{~>}[uull]_(.75){w'_{s1}}_(.45)x \\
\scriptscriptstyle\bullet\ar[rrrrrrr]^{\alpha_k}&&&&&&&\scriptscriptstyle\bullet \ar@{~>}[ul]_{v_{1k}}
}\]
a contradiction to the hypothesis that $C$ is constricted. Notice that $\epsilon \not = \gamma_k$ because $\epsilon \vert \rho_{\alpha_k}$ while $\gamma_k \not \vert \rho_{\alpha_k}$.
\item [(iii)] $x \vert w'_{s1}$ and $y \vert u_{1k}$. Then we have the following situation
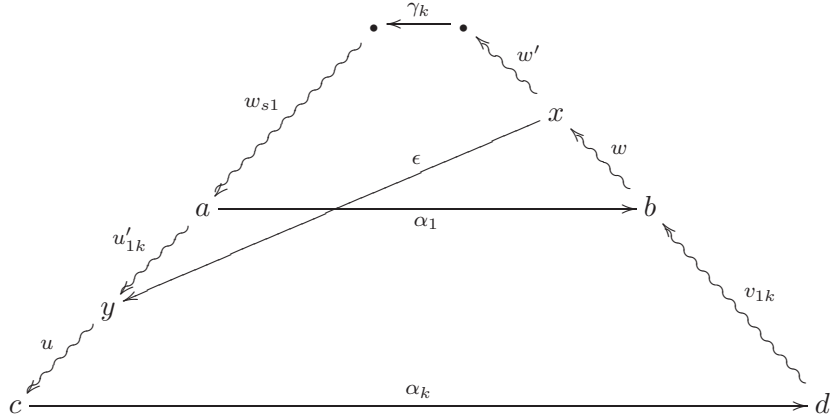
\[ \xymatrix{ &&&&\scriptscriptstyle\bullet\ar@{~>}[ddll]_{w_{s1}}&\scriptscriptstyle\bullet\ar[l]_{\gamma_k}\\
&&&&&&x   \ar@{~>}[ul]_{w'} \ar[ddlllll]_(.3)\epsilon \\
&&a \ar@{~>}[dl]_{u'_{1k}} \ar[rrrrr]_{\alpha_1}&&&&& b \ar@{~>}[ul]_{w} \\
& y  \ar@{~>}[dl]_u\\
c \ar[rrrrrrrrr]^{\alpha_k}&&&&&&&&& d \ar@{~>}[uull]_{v_{1k}}
}\]
where $u_{1k}= u'_{1k} u$, $w'_{s1} = w w'$   and $v= v_{1k} w$.
Because $C$ is constricted, there must be a relation in $C$ on the longer path from
$x$ to $y$. This relation together with the arrow $\epsilon$ yields a contradiction to Lemma \ref{2.4}. 

\item [(iv)] $x \vert v_{1k}$ and $y \vert w_{s1}$.  This case is symmetric to the previous one. 
\end{itemize}
This shows that we have no such terms in the sum and therefore $\lambda_{1k}=0$.  This completes the proof of the lemma.
\qed
\end{proof}

\medskip

\begin{corollary}
With the above notation, $\Hom_{C^e} (E_i, E_j) =0$ whenever $i \not = j$.
\end{corollary}

\begin{lemma}
For every $i$, we have that $\End_{C^e} E_i = k$.
\end{lemma}

\begin{proof}
Assume $E_i=<\alpha_1>$ and $[\alpha_1] = \{ \alpha_1, \cdots, \alpha_r\}$.  Because of Lemma \ref{45}, we have, for $\delta \in \End_{C^e} E_i$
$$\delta (\alpha_i) = \lambda_i \alpha_i$$
for some scalar $\lambda_i \in k$.
Now there exists a  strongly minimal relation containing the arrow $\alpha_1$, let it be
$$\rho = \sum_{j \in J} \mu_j w_j \alpha_j w'_j.$$
Applying $\delta$ yields
$$0 = \delta (\rho) = \sum_{j \in J} \mu_j \lambda_j w_j \alpha_j w'_j.$$
Subtracting from this the relation $\lambda_1 \rho$ we get
$$\sum_{j \in J\setminus \{ 1 \} } \mu_j (\lambda_j - \lambda_1)  w_j \alpha_j w'_j = 0.$$
Because $\rho$ is  strongly minimal, we get $\lambda_j = \lambda_1$, for every $j \in J\setminus \{ 1 \}$.  Because the arrow equivalence is transitive, we get $\lambda_j = \lambda_1$ for every $j \in \{1, \cdots, r\}$.
\qed
\end{proof}

\medskip

Now we are ready to prove Theorem \ref{31}.

\medskip

\begin{proof} 
If $C$ is hereditary, then $B=C$, $n_{B,C}=0$ and ${\HH}^1(B)={\HH}^1(C)$. If not, 
 assume first that $C$ is constricted.  Because ${\HH}^1(C,E) =0$, it suffices to prove that $\End_{C^e} E = k^n$.  But we have also shown that 
\[     \Hom_{C^e}(E_i, E_j) = \left \{ \begin{array}{lll}
0 & \mbox{if $i \not = j$}\\
k & \mbox{ if $i=j$.}
\end{array} \right. \]
This implies immediately the statement. 

Otherwise, $C$ is, up to duality, of one of the forms (i) (ii) of Lemma \ref{lem 5.2}. As observed in the proof of Lemma \ref{lem 5.3}, we have two distinct cases:
 
 (a) Either one can replace the non-constricted algebra $C$ by a constricted algebra $C'$ such that $n_{B,C}=n_{B,C'}$ and ${\HH}^1(C)\cong {\HH}^1(C')$ because of Lemma \ref{lem 5.*}. The statement then follows from the previous argument  applied to $B$ and $C'$.
 
 (b) Otherwise $B$ is, up to duality, of one of the forms (i) (ii) of Lemma \ref{lem 5.3}. Note that there exist several tilted algebras $C$ having $B$ as a relation-extension. However, because of Lemma \ref{lem 3.3}, the cardinality $|R|$ of a system of relations $R$ for each such tilted algebra $C$ is independent of the choice of $C$. Moreover, in this case, $n_{B,C}=|R|$, by Lemma \ref{lem 5.5}.
 
 Using Lemma \ref{lem 5.4}, it suffices to prove that, if $C$ is of the form (i), then ${\HH}^1(C)=k^2$ and, if $C$ is of the form (ii), then ${\HH}^1(C)=k$.
 This follows from another straightforward application of Happel's sequence. 
\qed
\end{proof}  

\medskip

\begin{example}
Let $C$ be the tilted algebra of euclidean type $\tilde{\mathbb{A}}_{2,2}$ given by the quiver
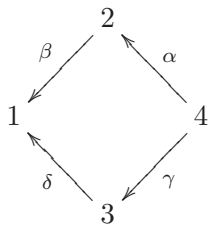
\[\xymatrix{&2\ar[ld]_\beta\\ 1&& 4\ar[lu]_\alpha\ar[ld]^\gamma\\ &3\ar[lu]^\delta}\]
bound by the relations $\alpha\beta=0$ and $\gamma\delta=0$.   The corresponding cluster-tilted algebra $B$ is given by the quiver
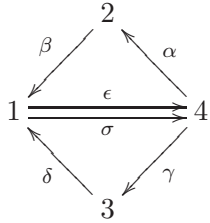
\[\xymatrix{&2\ar[ld]_\beta\\ 1\ar@<-2pt>[rr]_\sigma\ar@<2pt>[rr]^\epsilon&& 4\ar[lu]_\alpha\ar[ld]^\gamma\\ &3\ar[lu]^\delta}\]
bound by the relations $\alpha\beta=\beta\epsilon=\epsilon\alpha=0$ and $\gamma\delta=\delta\sigma=\sigma\gamma=0$. Note that $B$ is not schurian so the results from \cite{AR} cannot be used. 
The arrow equivalence class in this example consists of the two new arrows $\epsilon$ and $\sigma$, and therefore the relation invariant $n_{B,C}$ is equal to 2. Now Theorem \ref{31} implies that ${\HH}^1(B) \cong {\HH}^1(C)\oplus k^2\cong k^3.$

\end{example}

\medskip
{\ralf The following result has been proved for schurian cluster-tilted algebras in \cite[Corollary 3.4]{AR}. 
The statement is inspired from Skowro\'nski's famous question \cite[Problem 1]{S}: For
which algebras is   simple connectedness  equivalent to the vanishing of the first
Hochschild cohomology group?  }

\begin{theorem} Let $B=k \tilde Q / \tilde I$ be a cluster-tilted algebra.  Then ${\HH}^1(B)=0$ if and only if $B$ is hereditary with ordinary quiver a tree.
\end{theorem}

\begin{proof}
By \cite{ABS3}, the cluster repetitive algebra is a Galois covering of $B$ with infinite cyclic group $\Z$.  Moreover it is connected if and only if $B$ is not hereditary (because of \cite[1.4, Lemma 5]{ABS3}).  Assume thus that $B$ is not hereditary.  Because of the universal property {\ralf of the Galois covering}, there exists a group epimorphism
$$\pi_1(\tilde Q,\tilde I) \to \Z.$$
Let $k^+$ denote the additive group of the field $k$. The previous epimorphism induces a monomorphism of abelian groups $$\Hom (\Z, k^+) \to \Hom(\pi_1(\tilde Q, \tilde I), k^+)$$
which, composed with the canonical monomorphism 
 $\Hom(\pi_1(\tilde Q, \tilde I), k^+) \to HH^1(B)$
of \cite[Corollary 3]{PS}, yields a monomorphism $\Hom (\Z, k^+) \to HH^1(B)$.

Therefore, if $B$ is not hereditary, we have $HH^1(B) \not = 0$.  On the other hand, if $B$ is hereditary, then, because of \cite[1.6]{Ha2}, we have $HH^1(B)=0$ if and only if the quiver $\tilde Q$ of $B$ is a tree. \qed
\end{proof}

\section{The representation-finite case}\label{sect 6}

Throughout this section, let $B$ be a representation-finite cluster-tilted algebra. We present easy methods to compute the relation invariant $n_{B,C}$ and thus ${\HH}^1(B)$ in this case.
Let  $\QB$ be the
quiver of $B$  and let  $n$ be the number of points in $\QB$.

Choose  a tilted algebra $C$ such that $B=C\ltimes \Ext^2_C(DC,C)$. 
The \emph{number of relations} in $C$ is the dimension of
$\Ext^2_C(S_C,S_C) $, where $S_C$ is the sum of  a complete set of representatives of the isomorphism classes of simple $C$-modules. 
We say that a relation $r$ in $B$ is a \emph{new relation}
if it is not a relation in $C$.
It has been shown in \cite[Corollary 3.3]{AR} that in this case $n_{B,C}$ is equal to the number of relations in $C$
minus the number of new commutativity relations in $B$, and, moreover,
 \[ \textup{HH}^1(B)=k^{n_{B,C}}.\]
In particular, the integer $n_{B,C}$ does not depend on the choice of the tilted algebra $C$, and therefore we shall denote it in the rest of this section by $n_B$. 
The objective of this section is to show that one can read off the integer $n_B$ from
the quiver $\QB$ of $B$.

Recall from \cite{BGZ} that a \emph{chordless cycle} in $\QB$ is a full subquiver induced by a set
of points $\{x_1,x_2,\ldots,x_p\}$ which is topologically a cycle,
that is, the edges in the chordless cycle are precisely the edges
$x_i\,\frac{\quad}{}\,x_{i+1}$.

\begin{lemma}\label{lem 1}
The number of chordless cycles in $\QB$ 
is equal to the number of zero relations in $C$ plus twice the number
of commutativity relations in $C$.
\end{lemma}  
\begin{proof}
Consider the map
$\{ \textup{relations in } C\} \to \{\textup{new arrows in } B\}$
 that associates to a relation $\rho\in\Ext^2_C(S_i,S_j)$ the new arrow
 $\alpha(\rho):j\to i$. 
By  \cite[Corollary 3.7]{BR}, every chordless cycle 
 contains exactly one new arrow, and therefore it suffices to show that if $\rho$ is a
commutativity relation, then $\alpha(\rho)$ lies in
precisely two chordless cycles in $\QB$, and  if $\rho$ is  a zero
relation, then  $\alpha(\rho)$ lies in precisely one chordless cycle in $\QB$.

If $\rho$ is a commutativity relation, say $\rho=c_1-c_2$
where $c_1,c_2$ are paths from $i$ to $j$ in $\QC$, then  the
 concatenations  $\alpha(\rho) c_1$ and $\alpha(\rho) c_2$ are two chordless
 cycles. Then it follows from the fact that $\widetilde{Q}$ is a planar quiver (see \cite[Theorem A1]{CCS2}), that $\alpha(\rho)$  lies in precisely  two
 chordless cycles.

Otherwise, $\rho$ is a zero relation in $C$, and $\alpha(\rho)\rho $ is a chordless cycle
in $\QB$. We have to show that $\alpha(\rho)$ does not lie in two chordless
cycles. Suppose the contrary.
{Because of  \cite[Proposition 9.7]{FZ2}, every chordless cycle in $\QB$
is oriented. Therefore there exists another path $\rho'$ from $i$ to $
j$ in $\QB$ such that $\alpha( \rho) \rho'$ is a
chordless cycle. If $\rho'$ is also a path in $Q,$ then $\rho$ and $
\rho'$ are two parallel paths whose difference $\rho-\rho'$ is not
a relation in $C.$ This implies that the fundamental group of $C$ is
non-trivial, and this contradicts the well-known fact that tilted algebras
of Dynkin type are simply connected (see, for instance, \cite{L}). On the other
hand, if $\rho'$ is a path in $\QB$ but not in $Q$ then it
must contain at least one new arrow. But then the chordless cycle $\alpha
( \rho) \rho'$ contains two new arrows, a contradiction to
\cite[Corollary 3.7]{BR}.
}
\qed
\end{proof}  

\bigskip

An arrow in $\QB$ is called
\emph{inner arrow} if it is contained in two chordless cycles. Arrows
which are not inner arrows are called \emph{outer arrows}.
\begin{lemma}\label{lem 2}
The number of new inner arrows in $B$ is equal to the number of
commutativity relations in $C$.
\end{lemma}  

\begin{proof} Each commutavity relation in $C$ gives a new inner arrow in
  $B$. Conversely, suppose that $\alpha$ is a new inner arrow in $B$
  and let $ \rho,\rho'$ be the two paths in $\QB$ such that $\alpha
   \rho$ and $\alpha \rho'$ are the chordless cycles. By \cite[Corollary 3.7]{BR},
    $ \rho$ and $\rho'$ contain no new arrows, and hence $\rho$ and
  $\rho'$ are paths in $\QC$. Since the algebra $C$ is simply connected, it follows that
  $\rho-\rho'$ is a relation in $C$.
\qed
\end{proof}  

\begin{lemma}\label{lem 3} 
The number of old inner arrows in $B$ is equal to the number of
new commutativity relations in $B$.
\end{lemma}  

\begin{proof} 
We recall from \cite{CCS2, BMR2} the description of $B$ as a bound quiver algebra: For any arrow $\alpha$ in $\QB$ let
  $\mathcal{S}_\alpha$ be the set of paths $\rho$ in $\QB$ such that $\rho\alpha$
  is a chordless cycle and define 
\[\rho_\alpha=\left\{
\begin{array}{ll} \rho &\textup{if $\mathcal{S}_\alpha=\{\rho\}$}\\ 
 \rho-\rho' &\textup{if $\mathcal{S}_\alpha=\{\rho,\rho'\}.$}
\end{array}\right.\]
Let $I$ be the ideal in $k\QB$ generated by
the relations $ \cup_{\alpha\in(\QB)_1} \{\rho_\alpha\}$. Then
\[B=k\QB/I.\]

Because of the previous
remarks, commutativity relations are in bijection with inner arrows. If the
 relation is new, then the arrow is old and if the arrow is new { then}
 the relation is old.
\qed
\end{proof}  

\bigskip

We are now able to prove the main theorem of this section.

\begin{theorem}\label{thm 1} Let $B$ be a representation-finite cluster-tilted algebra and $\QB$ the quiver of $B$. Then
$
 n_B$ equals the  number of chordless cycles in $\QB$ minus the number
  of inner arrows in  $\QB.$
\end{theorem}

\begin{proof}
By definition, $n_B$ is the number of relations in $C$ minus the
number of new commutativity relations in $B$. By Lemmata \ref{lem 1} and \ref{lem 2}, the number of relations
in $C$ is equal to the number of chordless cycles minus the number of
new inner arrows in $\QB$. On the other hand, the number of new commutativity relations in $B$ is
equal to the number of { old} inner arrows in $\QB$, because of Lemma \ref{lem
  3}.
Therefore 

\[ n_B= \textup{ number of chordless cycles in }\QB - \textup{ number
  of inner arrows in } \QB.\]
\qed
\end{proof}

\begin{corollary}\label{cor 1} If $\QB$ is connected then
\[\begin{array}{rcl}
n_B&=&1+ \textup{ number of outer arrows in } \QB -n.\\
\end{array}  \]
\end{corollary}  

\begin{proof}
Because of
\cite[Theorem A1]{CCS2} the quiver $\widetilde{Q}$ is planar. In particular,
every arrow lies in at most two chordless cycles. Hence one can associate a simplicial complex on the
2-dimensional sphere to the quiver $\QB$, in such a way that $Q_0$ is the
set of points, $Q_1$ the set of edges and the set of chordless cycles
is the set of faces of the simplicial complex except the face coming
from the ``outside'' of the quiver (the unbounded component of the
complement when embedded in the plane). Using Euler's formula, we
see that the number of chordless cycles in $\QB$ is equal to
$1+|(\QB)_1|-|(\QB)_0|$, and then Theorem \ref{thm 1} yields
\[n_B= 1+ (|(\QB)_1|- \textup{ number
  of inner arrows in } \QB)-|(\QB)_0|, \] 
and the statement follows.
\qed
\end{proof}

\begin{obser}\label{rem 1} If $\QB$ is not connected then
\[ 
\begin{array}{rcl}
  n_B&=& \textup{number of connected components of } \QB\\ && + \textup{
  number of outer arrows in } \QB -n.
\end{array}  \]
\end{obser}

As an application, we show the following corollary on deleting points.
Let $x\in (\QB)_0$, and $e_x\in B$ the associated idempotent. Then
$B/Be_xB$ is cluster-tilted and the quiver of $B/Be_xB$ is obtained
from $\QB$ by deleting the point $x$ and all arrows adjacent to $x$, see \cite[Section 2]{BMR3}.
Define the \emph{Hochschild degree} of $x$ to be the integer
\[\textup{deg}_{{\rm HH}}(x)= n_B-n_{B/Be_xB}\]

\begin{corollary}\label{cor 2}
\[\begin{array}{rcl}
\textup{deg}_{\rm{HH}}(x)&=& 
\textup{ number of chordless cycles going through } x\\
&&- \textup{ number of inner arrows on the chordless cycles going through } x
\end{array}  \]
\end{corollary}  

\begin{proof} Using Theorem \ref{thm 1}, we get 
that $\textup{deg}_{\rm{HH}}(x)$ is equal to 
the number of chordless cycles that are adjacent to $x$ minus the number of
inner arrows in $\QB$ plus the number of inner arrows in
$Q_{B/Be_xB}$. Now $\alpha$ is an inner arrow in $\QB$ which is not an inner
arrow in $Q_{B/Be_xB}$, precisely if $\alpha$ lies on
two chordless cycles in $\QB$ at least one of which goes through $x$.
\qed
\end{proof}  

\begin{example} 
The following quiver is the quiver of a cluster-tilted algebra  of type $\mathbb{E}_8$.
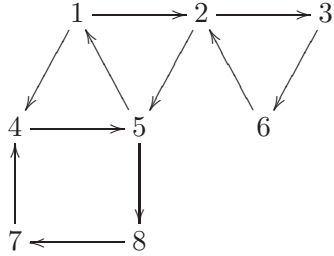
\begin{center}
$$\xymatrix@R=12pt@C=12pt{
&1\ar[rr]\ar[ldd]&&2\ar[rr]\ar[ldd]&&3\ar[ldd]\\ \\
4\ar[rr]&&5\ar[luu]\ar[dd]&&6\ar[luu]\\
\\
7\ar[uu]&&8\ar[ll]
}$$
\end{center}
The quiver has $4$ chordless cycles and $2$ inner arrows, so Theorem \ref{thm 1} yields
$$\textup{HH}^1(B)=k^{4-2}=k^2.$$
On the other hand, the quiver has $9$ outer arrows, so, using Corollary \ref{cor 1}, we also get
$$\textup{HH}^1(B)=k^{1+9-8}=k^2.$$ 
The point $2$ has Hochschild degree $2-1=1$, by Corollary \ref{cor 2}. So
$\textup{HH}^1(B/Be_2B)=k$.
The quiver of $B/Be_2B$ is the following.

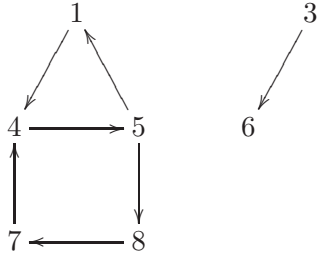
\begin{center}
$$\xymatrix@R=12pt@C=12pt{
&1\ar[ldd]&&&&3\ar[ldd]\\ \\
4\ar[rr]&&5\ar[luu]\ar[dd]&&6\\
\\
7\ar[uu]&&8\ar[ll]
}$$
\end{center}
{ Observe that Remark \ref{rem 1} applies here: the number of connected components is $2$, the number of outer arrows is $6$, and $n=7$.  Thus we get $n_B= 2+6-7=1$.}
\end{example}

\footnotesize{
\noindent Ibrahim Assem: \\
D\'epartement de Math\'ematiques, Universit\'e de Sherbrooke, \\ Sherbrooke, Qu\'ebec, Canada, J1K 2R1. \\
{\it ibrahim.assem@usherbrooke.ca} \\

\noindent Mar\'\i a Julia Redondo:
\\Departamento de Matem\'atica,
Universidad Nacional del Sur,\\Av. Alem 1253\\8000 Bah\'\i a Blanca,
Argentina.\\ {\it mredondo@criba.edu.ar} \\

\noindent Ralf Schiffler:
\\ Department of Mathematics, University of Connecticut,\\ Storrs, CT, USA 06268 \\{\it schiffler@math.uconn.edu}}

\end{document}